\documentclass[12pt]{amsart}
\usepackage{amssymb}
\usepackage{amsmath}
\usepackage{xspace}
\usepackage[all]{xypic}
\usepackage{amscd}
\usepackage{enumerate}
\usepackage{color}
\usepackage{graphicx}
\usepackage{algorithm}
\usepackage{fixme}
\usepackage[noend]{algorithmic}
\theoremstyle{plain}
\newtheorem{theorem}{Theorem}[section]
\newtheorem{corollary}[theorem]{Corollary}
\newtheorem{proposition}[theorem]{Proposition}
\newtheorem{lemma}[theorem]{Lemma}
{\theoremstyle{remark}

\newtheorem{remark}[theorem]{Remark}}
{\theoremstyle{definition}
\newtheorem{definition}[theorem]{Definition}
\newtheorem{notation}[theorem]{Notation}
\newtheorem{example}[theorem]{Example}
}
\numberwithin{equation}{section}

\newcommand{\TKQ}[1]{}%{\marginpar{\color{red}{\small #1}\color{black}}}
\newcommand{\SEQ}[1]{}%{\marginpar{\color{blue}{\small #1}\color{black}}}
\newenvironment{proofof}[1]{\noindent\emph{Proof of #1.}}{\hfill$\Box$}
\newcommand{\EF}{F^{01}} %%S^F_H}

\newcommand{\myi}{E_{\text{\sc{i}}}}
\newcommand{\myii}{E_{\text{\sc{ii}}}}
\newcommand{\myiii}{E_{\text{\sc{iii}}}}
\newcommand{\myv}{E_{\text{\sc{iv}}}}
\newcommand{\myvi}{E_{\text{\sc{v}}}}
\newcommand{\myvii}{E_{\text{\sc{vi}}}}
\newcommand{\myviia}{E_{\text{\sc{vi},a}}}
\newcommand{\myviib}{E_{\text{\sc{vi},b}}}
\newcommand{\N}{\mathbb{N}}
\newcommand{\Z}{\mathbb{Z}}

\newcommand{\C}{\mathbb{C}}
\newcommand{\T}{\mathbb{T}}

\newcommand{\bK}{{\mathbb K}}

\newcommand{\K}{\mathbb{K}}

\newcommand{\F}{{\mathcal F}}

\newcommand{\reg}{_{{\operatorname{rg}}}}
\newcommand{\sing}{_{\operatorname{sg}}}
\renewcommand{\phi}{\varphi}

\newcommand{\Ca}{$C^*$-al\-ge\-bra }
\newcommand{\CA}{$C^*$-al\-ge\-bra}

\newcommand{\shom}{$*$-ho\-mo\-mor\-phism }
\newcommand{\shoM}{$*$-ho\-mo\-mor\-phism}

\newcommand{\makestandard}[2]{\widehat{#1}^{#2}}
\newcommand{\g}[1]{\chi_{#1}}
\newcommand{\idealof}{\lhd}
\newcommand{\gidealof}{\blacktriangleleft}

\newcommand{\toep}[1]{\mathcal T(#1)}
\newcommand{\nap}{admissible pair\xspace}
\newcommand{\naps}{admissible pairs\xspace}

\newcommand{\CK}[1]{(CK#1)}

\newcommand{\mys}{\mathsf S}
\newcommand{\myr}{\mathsf R}

\newcommand{\pv}[2]{p_{#1,#2}}
\newcommand{\pVX}{\pv{V}{\XV}}
\newcommand{\XV}{X} %{\{X_v\}_{v\in V}}
\newcommand{\pWY}{\pv{W}{\YW}}
\newcommand{\YW}{Y} %{\{Y_w\}_{w\in W}}
\newcommand{\ideal}[1]{I(#1)}
\newcommand{\infset}[1][E]{{#1}_{\infty,\mys}^0}

\begin{document}
\title[Semiprojectivity and properly infinite projections]{Semiprojectivity and properly infinite projections in graph $C^*$-algebras}
\date{\today}
\author{S\o ren EILERS}
\address{S\o ren Eilers, Department of Mathematical Sciences\\University of Copenhagen\\DK-2100 Copenhagen \O\\DENMARK
}
 \email{eilers@math.ku.dk}

\author{Takeshi KATSURA}
\address{Takeshi Katsura, Department of Mathematics\\ Keio University\\
Yokohama, 223-8522\\ JAPAN}
\email{katsura@math.keio.ac.jp}
\thanks{This work was supported by the Danish National Research Foundation through the Centre for Symmetry and Deformation (DNRF92). The first named author was further supported by  VILLUM FONDEN through the network for Experimental Mathematics in Number Theory, Operator Algebras, and Topology.}


\begin{abstract}
We give a complete description of which unital graph $C^*$-algebras are semiprojective, and use it to disprove two conjectures by Blackadar. To do so, we perform a detailed analysis of which projections are properly infinite in such $C^*$-algebras.
\end{abstract}
 %\SEQ{More}

\maketitle

\section{Introduction}
The semiprojective $C^*$-algebras (cf.\ \cite{bb:stc}, \cite{tal:lsppc}) are of paramount importance in the study of classification and structure of $C^*$-algebras, but in spite of decades of interest many basic questions about this class remain open.  In this paper, we study semiprojectivity in the context of graph $C^*$-algebras and completely resolve the question of when a unital such algebra is semiprojective. 

The theory of the class of graph $C^*$-algebras plays a prominent role in our understanding of  semiprojectivity, and in fact, until a recent announcement by Enders (\cite{de:blackadar}), 
the only known examples of nuclear and simple $C^*$-algebras were corners of graph algebras. Indeed, as detailed below,  a complete understanding of semiprojectivity for unital graph $C^*$-algebras very easily leads to the resolution, in the negative, of two long-standing conjectures by Blackadar concerning closure properties of the class of semiprojective $C^*$-algebras.

The fact that all simple and unital graph algebras are semiprojective was proved by Szyma\'nski by a refinement of methods in Blackadar's proof that $\mathcal O_\infty$ is semiprojective. Refining the argument further, Spielberg extended this result to all simple graph algebras with finitely generated $K$-theory. Since all of these results draw on the fact that in  purely infinite $C^*$-algebras all projections are properly infinite it is not surprising that our analysis of the general unital case draws heavily on the notion of properly infinite projections. In fact, we will prove in Theorem \ref{Thm:finite} and Theorem \ref{Thm:finiteII}

\begin{theorem}\label{main}
Let $E$ be a graph for which the vertex set $E^0$ is finite. 
For each $v\in E^0$, set
\[
\Omega_v=\{w\in E^0\mid \text{There are infinitely many edges from }w\text{ to } v\}.
\]
The following are equivalent
\begin{enumerate}[(i)]
\item $C^*(E)$ is semiprojective 
\item $C^*(E)$ is weakly semiprojective 
\item $C^*(E)$ is weakly semiprojective w.r.t. the class of unital graph $C^*$-algebras
\item  For each $v$, $\Omega_v$
does not have (FQ) 
\item  For each $v$, $p_{\Omega_v}$ is properly infinite.
\item Any corner $p C^*(E)p$ with $p\in C^*(E)$ is semiprojective
\item For any gauge-invariant ideal $I\gidealof C^*(E)$, any corner $p(C^*(E)/I)p$ with $p\in C^*(E)/I$ is semiprojective
\item No pair of gauge-invariant ideals $I\gidealof J\gidealof C^*(E,\mys)$ has $J/I$ Morita equivalent to $\K^\sim$ or $(C(\T)\otimes \bK)^\sim$
\end{enumerate}
\end{theorem}

We will return to the property (FQ) mentioned in (iv) below, and draw first the reader's attention to condition (v) which reduces the issue of semiprojectivity to the question of whether or not certain projections are properly infinite. Note also that property (viii) reduces the \emph{a priori} global question of semiprojectivity to the local question of having a subquotient of a certain size to be of a special form. Since in general semiprojectivity passes neither to ideals or quotients, this is a highly surprising result which we discovered only after substantial computer experimentation had convinced us that is was likely true. The other statements compare semiprojectivity to weak semiprojectivity and semiprojectivity of corners, proving that 
all these notions coincide.

To prove Theorem \ref{main}, and to provide a concrete test for semiprojectivity which may be completely automated, we must study the question of which projections in graph algebras are properly infinite. To answer this question we must resolve the relation between various notions of infinity of projections as studied by R\o rdam, and it turns out that in this case many of the general complications shown by him to exist even in simple $C^*$-algebras  do not occur. In fact, we provide the following dichotomies

\begin{theorem}\label{dicho}
Let $p\in C^*(E)$.
\begin{enumerate}[(i)]
\item Either $p$ is infinite, or for any surjection $\pi:C^*(E)\to A$, $\pi(p)$ is stably finite
\item Either $p$ is properly infinite, or for some surjection $\pi:C^*(E)\to A$, $\pi(p)$ is nonzero and stably finite
 \end{enumerate}
\end{theorem}

\noindent
which hold true also in the non-unital case. We prove these results in Proposition \ref{dichoI} and Theorem \ref{Thm:notPI:pre} below, establishing first (i) by a concrete criterion for testing infinity of projections $p_V$ associated to finite sets of vertices  $V\subseteq E$, and  extending our findings to general projections using the recent understanding of the non-stable $K$-theory of graph algebras in \cite{pamamep:nkga} and  \cite{dhmlmmer:nklpa}. With this, (ii) follows by appealing to the general description of proper infiniteness obtained by Kirchberg and R\o rdam. However, in order to perform the necessary analysis of semiprojectivity, a concrete criterion for which vertex projections $p_V$ are properly infinite is required, and drawing on (ii) above we provide the criterion that $V$ must not have the condition (FQ) to the effect that none of the configurations illustrated in Figure \ref{fqfig} may be found. The details will be given below, but the general idea of the three kinds of configurations which are indicative of the failure of proper infiniteness is that one or several vertices $w$ or $w_i$ have the property that save a possible unique loop based on $w$ in the subscase (TQ), there is at least one, but at most a finite number of paths starting at $w$ and ending in $V$. In the subcase (MQ) one further must require that the $w$ is a singular vertex of the graph, and in the subcase (AQ) a whole sequence of such $w_i$ must be provided.
\begin{figure}
\begin{center}
\begin{tabular}{ccc}
\makebox[3.2cm][t]{\includegraphics[width=3cm]{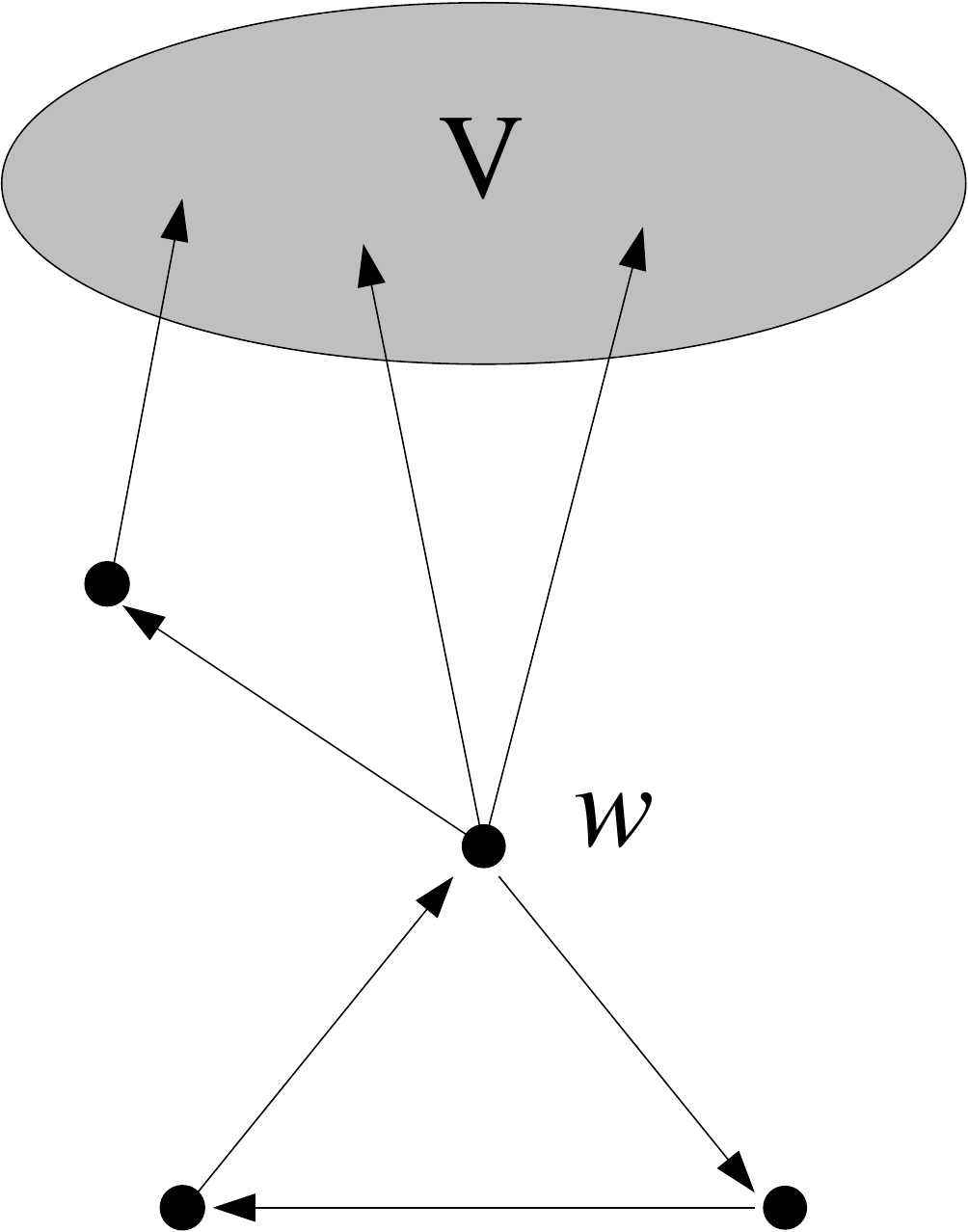}\rule{0mm}{3cm}}&\includegraphics[width=3cm]{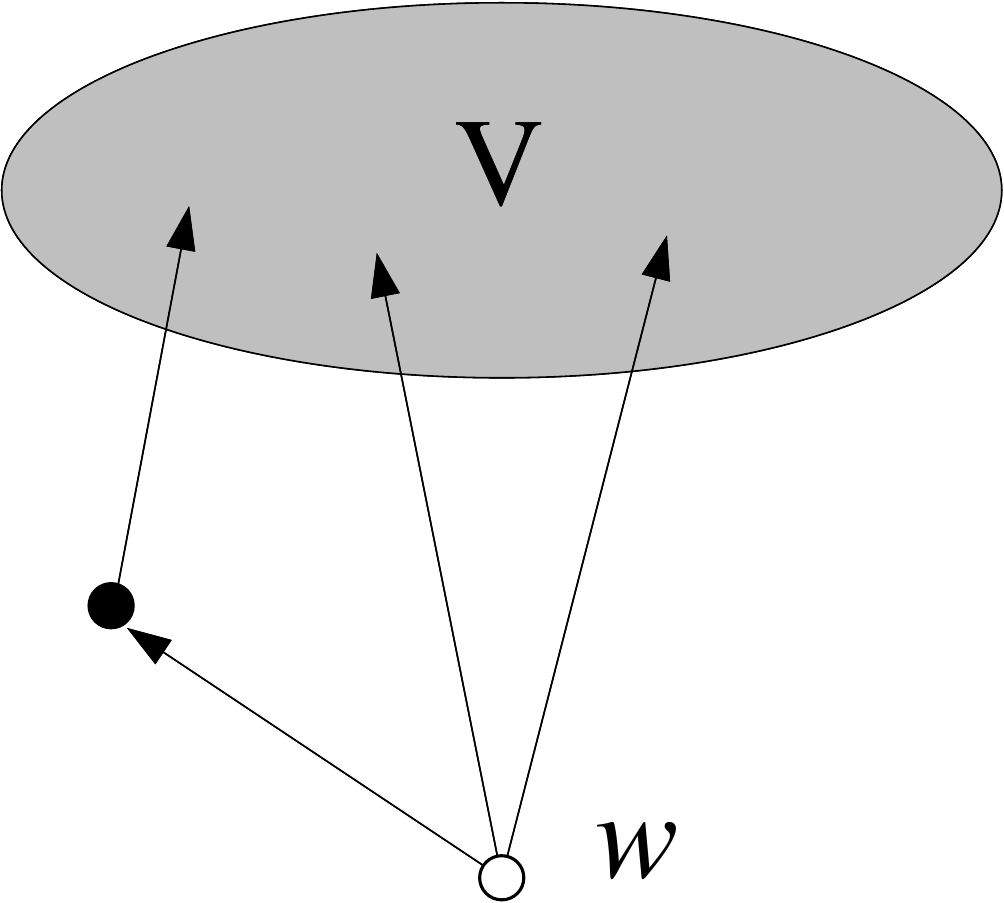}&\includegraphics[width=3cm]{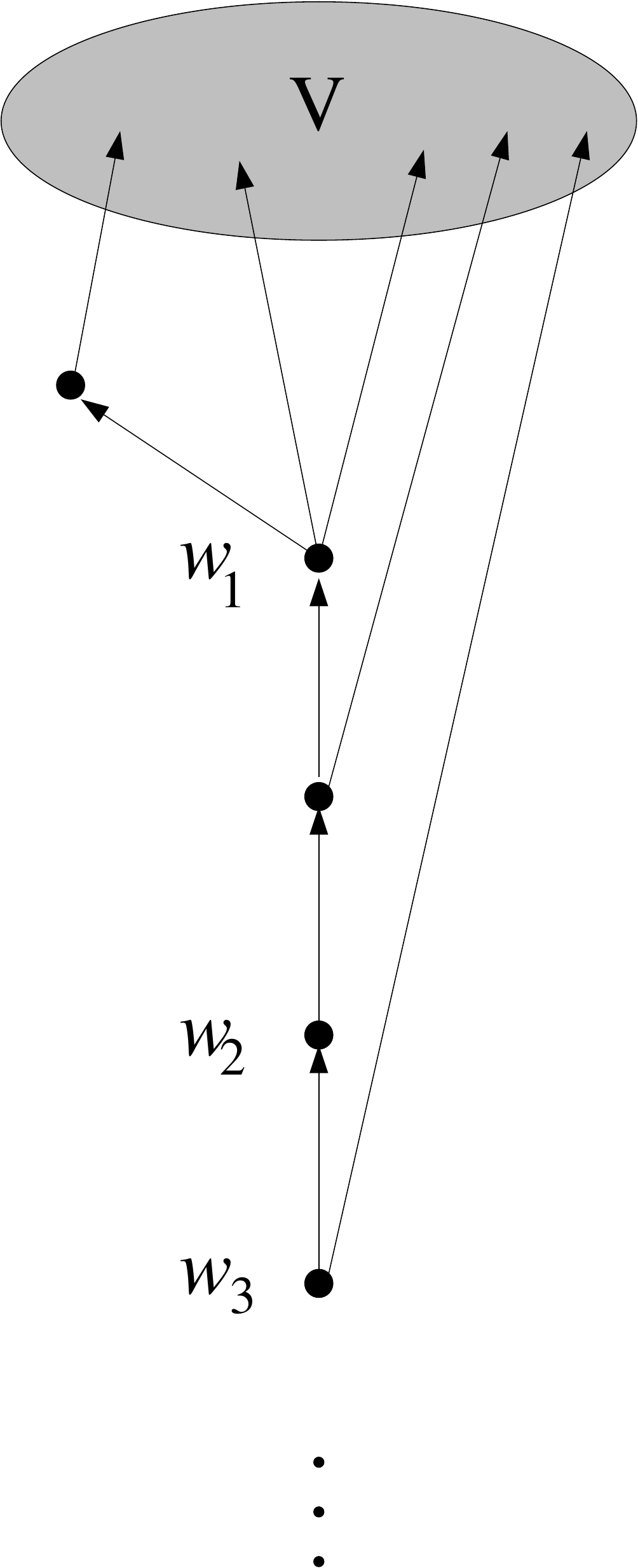}\\
$(TQ)$&$(MQ)$&$(AQ)$
\end{tabular}
\end{center}
\caption{The three cases leading to condition (FQ)}\label{fqfig}
\end{figure}

Finding these configurations in a graph will enable us to find quotients of the $C^*$-algebra in question in which the projection defined by $V$ defines a corner which is either a circle algebra in the case of (TQ), a matrix algebra in the case of (MQ), or an AF algebra in the case of (AQ). As we shall see, this detailed understanding of how a projection may fail to be properly infinite is the key to understanding semiprojectivity in the algebras in question.

Blackadar conjectured at the Kyushu conference in 1999 (cf.\ \cite[4.5]{bb:ssc}) that  the following two related claims concerning permanence of the class of semiprojective $C^*$-algebras are true
\begin{enumerate}[(I)]
\item When an extension
\[
\xymatrix{
{0}\ar[r]&{I}\ar[r]&{A}\ar[r]&{\C}\ar[r]&0}
\]
splits, and $I$ is semiprojective, then so is $A$.
\item When $B$ is a full corner in $D$, and $D$ is semiprojective, then so is $B$.
\end{enumerate}
These conjectures have commanded a lot of attention and both are known to hold true in many interesting special cases. Based on our results, we provide counterexamples to both conjectures in Corollary \ref{bbwaswrongagain} and Example  \ref{bbwasagainnotentirelyright}, respectively.

\section{Preliminaries}\label{prelims}

\subsection{Relative graph algebras}\label{SecRGA}

A graph 
$E=(E^0,E^1,s,r)$ is given as usual by its vertex set $E^0$, its edge
set $E^1$, and range and source maps $r,s:E^1\to E^0$.

\begin{notation}
We set
\begin{align*}
vE^1&=\{e\in E^1\mid r(e)=v\}=r^{-1}(\{v\})\\
E^1w&=\{e\in E^1\mid s(e)=w\}=s^{-1}(\{w\})\\
vE^1w&=vE^1\cap E^1w
\end{align*}
We denote by $E^*$ the set of paths in $E$, including $E^0$ which we think of as paths of length 0, and similarly denote
\begin{align*}
vE^*&=\{\alpha \in E^1\mid r(\alpha )=v\}\\
E^*w&=\{\alpha\in E^*\mid s(\alpha)=w\}\\
vE^*w&=vE^*\cap E^*w
\end{align*}
In a few instances we need to consider only $E^\dagger=E^*\setminus E^0$ and define $vE^\dagger$, $E^\dagger w$, $vE^\dagger w$ similarly. Note that $s(vE^*)$ contains the vertices that emit a path to $v$ and that $r(E^*v)$ contains the vertices that receive a path from $v$. We extend this notation for subsets $V,W\subseteq E^0$ to obtain sets of edges \[
VE^1,E^1W,VE^1W,VE^*,E^*W,VE^*W,VE^\dagger,E^\dagger W,VE^\dagger W\] in the obvious way.
\end{notation}

\begin{definition}
For a graph $E=(E^0,E^1,s,r)$, 
a {\em source} is a vertex $v$ with $vE^1=\emptyset$, and an {\em infinite receiver} is a vertex $v$ with $|vE^1|=\infty$. We call sources and infinite receivers {\em singular} vertices and denote  the set of singular vertices $E^0\sing$. The remaining vertices are 
said to be {\em regular} and denoted $E^0\reg$.
\end{definition}

\begin{definition}[Cf. \cite{psmmt:atc}]
Let $E=(E^0,E^1,s,r)$ be a graph, 
and $\mys$ be a subset of $E^0\reg$.
We define the {\em relative graph algebra} $C^*(E,\mys)$ 
to be the universal \Ca 
generated by mutually orthogonal projections $\{p_v\}_{v \in E^0}$ 
and partial isometries $\{s_e\}_{e \in E^1}$ such that 
\begin{enumerate}[(CK1)]
\item $s_e^*s_e=p_{s(e)}$ for $e \in E^1$, 
\item $s_es_e^* \leq p_{r(e)}$ for $e \in E^1$,
\item $\sum_{e\in vE^1}s_es_e^* =p_v$ for $v \in \mys$. 
\end{enumerate}
\end{definition}

Note that \CK{3} makes sense precisely at $v\in E^0\reg$, where $0<|vE^1|<\infty$. Thus it is crucial that $\mys\subseteq   E^0\reg$ for the definition to make sense.

\begin{notation}
We will draw graphs indicating with ``$\bullet$'' vertices in $\mys$ and with ``$\circ$'' vertices not in $\mys$. If there is an infinite number of edges between two vertices we write ``$\xymatrix{{\bullet}\ar@{=>}[r]&{\circ}}$'' or 
``$\xymatrix{{\circ}\ar@{=>}[r]&{\circ}}$'' depending upon whether or not the emitting vertex is in $\mys$. Note that the receiving vertex is singular, and hence must be marked ``$\circ$''.

For instance, consider the graph
\[
\xymatrix{
{\circ}\ar[r]\ar[dr]&{\bullet}\\
{\circ}\ar@{=>}[r]\ar@(ul,dl)[]\ar@(r,u)[]&{\circ}\ar@<+0.1mm>@/^/[u]\ar@<-0.1mm>@/_/[u]}
\]
The vertices on the top left and the bottom right must be outside of $\mys$ because they are singular, but at the other two vertices we have made a choice, indicated as described above.
\end{notation}

\begin{definition}
$C^*(E,E^0\reg)$ is simply denoted by $C^*(E)$ 
and called the {\em graph algebra} of $E$. 
$C^*(E,\emptyset)$ is called the {\em Toeplitz algebra} of
$E$ and denoted $\toep{E}$.
\end{definition}

The name Toeplitz algebra is derived from the fact that the $C^*$-algebra given by 
\[
\xymatrix{\circ\ar@(ld,lu)[]^(1.1){v_T}^-{e_T}}
\]
is in fact the standard Toeplitz algebra, as is easily seen. 
Graph algebras were defined in \cite{KPRR} and the Toeplitz algebras have mainly been considered as tools for their exploration, cf.\ \cite{tmcsemt:imkga} and \cite{tk:cac}. We are interested in the
intermediate cases where \CK{3} is imposed at some regular vertices,
but not all, for two reasons. First, because of the way the relative
graph algebras interpolate between the graph algebras and the Toeplitz
algebras, they form a natural class in which to work. Second, we will need
to work extensively with sub-$C^*$-algebras of relative graph algebras associated
to subgraphs, and these mostly fail to be graph algebras of subgraphs even when we start with a graph algebra 
$C^*(E)$, as explained in Lemma \ref{subgraph} below.

Any relative graph $C^*$-algebra $C^*(E,\mys)$ is in fact a standard graph algebra $C^*(\makestandard{E}{\mys})$ where $\makestandard{E}{\mys}$ is the graph with
\begin{gather*}
(\makestandard{E}{\mys})^0=E^0\sqcup \{v'\mid v\in E^0\reg\setminus \mys\}\\
(\makestandard{E}{\mys})^1=E^1\sqcup \{e'\mid v\in (E^0\reg\setminus \mys)E^1E^0\}
\end{gather*}
with $r(e')=r(e)$ and $s(e')=s(e)'$; in words, for each regular vertex not in $\mys$ we add a new vertex, and for each edge emitted from such a vertex, a copy is made starting at the new vertex and ending  at the old.  Thus there is no added generality obtained by working with relative graph algebras, but as we shall see it will make our statements quite a lot  cleaner. 

\begin{example}\label{threepoints}
The reader may find it instructive to find the graph algebra description of
\[
\xymatrix{
\circ\ar[r]&\circ\ar[r]&\circ}
\]
and check that the corresponding $C^*$-algebra is $\C\oplus M_2(\C)\oplus M_3(\C)$.
\end{example}

Relative graph $C^*$-algebras come equipped with a \emph{gauge action}
$\gamma:\mathbb{T}\to\operatorname{Aut}(C^*(E,\mys))$ defined by
\[
\gamma_z(p_v)=p_v\qquad \gamma_z(s_e)=zs_e
\]
We say that an ideal is \emph{gauge invariant} when it is fixed by
$\gamma$, and write $I\gidealof C^*(E,\mys)$ is this case. A key application of this notion is the \emph{gauge invariant uniqueness theorem}, which in our setting takes the form described below.

Note that in $C^*(E,\mys)$, all projections of the form $p_v$ or 
\[
p_v-\sum_{e\in X}s_es_e^*
\]
with $v\not \in\mys$ or $X\subsetneq vE^1$ are nonzero. This follows by the universal property of $C^*(E,\mys)$ since standard constructions yield a Hilbert space representation of \CK{1}--\CK{3} where these projections are not zero. In the other direction, the gauge invariant uniqueness theorem says that any gauge invariant map which does not annihilate these projections is in fact injective:

\begin{theorem}\label{giut}
  When $A$ has a circle action $\beta_z$, and $\phi:C^*(E,\mys)\to A$ is given with
\begin{enumerate}[(i)]
\item $\phi$ is equivariant (i.e. $\beta_z\circ \phi=\phi\circ \gamma_z$)
\item For any $v\in E^0$, $\phi(p_v)\not=0$
\item For any $v\in E^0\reg,$ when $X\subseteq vE^1$ and
\[
\phi\left(p_v-\sum_{e\in X}s_es_e^*\right)=0
\]
then $v\in \mys$ and $X=vE^1$,
\end{enumerate}
then $\phi$ is injective.
\end{theorem}
\begin{proof}
This follows from the standard gauge invariant uniqueness theorem (\cite{bhrs:iccig}) by replacing $C^*(E,\mys)$ with $C^*(\makestandard{E}{\mys})$ for $\makestandard{E}{\mys}$ as described above. The condition (iii) ensures that $\phi(p_{v'})\not=0$ for the added vertices in $\makestandard{E}{\mys}$.
\end{proof}

\begin{definition}
Let $E=(E^0,E^1,s,r)$ be a graph. A \emph{subgraph} of $E$ is given by $F^0\subseteq E^0$ and $F^1\subseteq
E^1$ such that 
\[
r(F^1)\subseteq F^0\qquad s(F^1)\subseteq F^0
\]
In this case, $(F^0,F^1,s,r)$ is itself a graph which we denote $F$.
\end{definition}

We set
\[
\EF=\{v\in F^0\mid vE^1\subseteq F^1\};
\]
in words, this set contains all vertices in $F$ receiving only edges in $F^1$.

\begin{lemma}\label{subgraph}
Let $F$ be a subgraph of $E$ and $\mys\subseteq E^0\reg$ a
subset. The relative graph $C^*$-algebra $C^*(F,\mys\cap \EF)$ is
canonically isomorphic to the sub-$C^*$-algebra of $C^*(E,\mys)$
generated by $\{p_v\}_{v\in F^0}\cup\{s_e\}_{e\in F^1}$.
\end{lemma}
\begin{proof}
We may define a map $\phi:C^*(F,\mys\cap \EF)\to C^*(E,\mys)$ by mapping
generators to generators, since the elements in $C^*(E,\mys)$ satisfy the relevant conditions. Indeed, conditions \CK{1} and \CK{2} are inherited, and 
the relevant instances of \CK{3} hold true since when $v\in \mys\cap \EF$, we have that
\[
p_v=\sum_{e\in vE^1}s_es_e^*=\sum_{e\in vF^1}s_es_e^*
\]

\end{proof}

For a finite subset $V$ of $E^0$, 
we set $p_V=\sum_{v \in V} p_v \in C^*(E,\mys)$. This is a projection
since all the $p_v$ are orthogonal. 

\begin{definition}\label{defnap}
We say a pair $(H,\myr)$ of subsets $H,\myr\subseteq E^0$ is an \emph{\nap} when
\begin{enumerate}[(i)]
\item $H = \{ v \in \myr \mid vE^1 \subseteq E^1H \}$
\item $| vE^1 \setminus E^1H | < \infty$ for all $v \in \myr$.
\end{enumerate}
\end{definition}

We can (cf.\ \cite{tk:iscac}) determine the structure of 
gauge-invariant ideals of the relative graph algebra $C^*(E,\mys)$ by \naps as follows: 

\begin{theorem}
The gauge invariant ideals of the relative graph algebra $C^*(E,\mys)$
are in one-to-one correspondence with \naps $(H,\myr)$ such that
$\mys\subseteq \myr$. When $I$ is the ideal given by $(H,\myr)$, there is a natural isomorphism
\[
C^*(E,\mys)/I\simeq C^*(E\setminus H,\myr\setminus H)
\]
\end{theorem}
\begin{proof}
Follows from the standard description of gauge invariant ideals of graph algebras  (\cite{bhrs:iccig}) by replacing $C^*(E,\mys)$ with $C^*(\makestandard{E}{\mys})$ for $\makestandard{E}{\mys}$ as described above, and instead of $(H,B)$ with $B$ the breaking vertices taking $(H,E^0\reg\cup B)$.
\end{proof}

The reader is asked to note that part  (i) of Definition \ref{defnap} implies that $H\subseteq \myr$, and that
\begin{enumerate}[(I)]
\item  $H$ is \textbf{hereditary},
i.e. that  $HE^1 \subseteq E^1H$ (in words, that it is impossible to enter $H$ from the outside).
\item $H$ is \textbf{$\myr$-saturated}, i.e. that $vE^1\subseteq HE^1$ for $v\in \myr$ only when 
$v\in H$ (in words, that if $v\in \myr$ receives only from $H$, then $v$ is already in $H$)
\end{enumerate}
Part (ii) of Definition \ref{defnap} is a finiteness condition which is equivalent to the well-known concept of \textbf{breaking vertices} for graph algebras, but is much more convenient when working with relative graph algebras. Let us look a few examples.

\begin{example}\label{recexx}
We consider graphs $\myi$, $\myii$, and $\myiii$ given by
\[
\raisebox{-1cm}[1.1cm]{}
\xymatrix{
{\bullet}&{\circ}\\
\bullet\ar@(lu,ld)[]\ar@(l,d)[]\ar[u]_(0.2){w_0}_(0.9){v_0}
&\circ\ar[l]\ar@{=>}[u]^(0.2){w_1}^(0.9){v_1}
}
\qquad
\xymatrix{
{\circ}&{\circ}\\
\bullet\ar@(lu,ld)[]\ar@(l,d)[]\ar@{=>}[u]_(0.2){w_0}_(0.9){v_0}&\circ\ar[l]\ar@{=>}[u]^(0.2){w_1}^(0.9){v_1}
}
\qquad
\xymatrix{
{\bullet}\ar@(l,u)[]\ar@(ul,ur)&{\circ}\\
\bullet\ar@(lu,ld)[]\ar@(l,d)[]\ar[u]_(0.2){w_0}_(0.9){v_0}&\circ\ar[l]\ar@{=>}[u]^(0.2){w_1}^(0.9){v_1}
}
\bigskip
%\qquad
%\xymatrix{
%\bullet\ar@(ld,lu)[]\ar@(l,u)[]&{\circ}\\
%\bullet\ar@(lu,ld)[]\ar@(l,d)[]\ar[u]_(0.2){w_0}_(0.9){v_0}&\circ\ar[l]\ar@{=>}[u]^(0.2){w_1}^(0.9){v_1}
%}
\]
as well as $\myv$, $\myvi$, and $\myvii$ given by
\[
%\raisebox{-1cm}[1.1cm]{}
\rule[-1cm]{0cm}{2cm}
\xymatrix@C-=4mm{
&{\circ}&\\%\ar@(ul,ur)[]\ar@(l,u)&\\
\bullet\ar@(lu,ld)[]\ar@(l,d)[]\ar[ur]_(0.3){w_0}^(1.1){v_0}&&\circ\ar[ll]\ar@{=>}[ul]^(0.3){w_1}
}
\qquad
\xymatrix@C-=4mm
{
&{\circ}\ar@{=>}@(ul,ur) []&\\
\bullet\ar@(lu,ld)[]\ar@(l,d)[]\ar[ur]_(0.3){w_0}&&\circ\ar[ll]\ar@{=>}[ul]^(0.3){w_1}_(0.9){v_0}
}
\qquad
\xymatrix@C-=4mm{
&{\circ}\ar@(ul,ur)[]^(0.9){v_0}\ar@(l,u)&\\
\bullet\ar@(lu,ld)[]\ar@(l,d)[]\ar[ur]_(0.3){w_0}&&\circ\ar[ll]\ar@{=>}[ul]^(0.3){w_1}
}
\bigskip
\]
The choice of $\mys$ indicated is in each case identical to the set of regular vertices, so that the relative graph algebras are in fact graph algebras. In each case, let us consider $H=\{w_0,w_1\}$ noting that it is always hereditary, and consider our options for choosing $\myr$ such that $(H,\myr)$ becomes an \nap\ with $\mys\subseteq \myr$. For all of the graphs $\myi$--$\myiii$, if we were to include $v_1$ in $\myr$, $H$ would not be $\myr$-saturated. Similarly, $v_0\not\in \myr$ for $\myi$, but since in this case $v_0\in \mys$, there is no way to complement $H$ as desired in this case. For $\myii$ we can take $\myr=H$ and for $\myiii$ we can take $\myr=H\cup\{v_0\}$, and these are the unique choices. 

Turning to the similar question in $\myv-\myvii$ we see that $\myr=H$ is the only choice for $\myv$ to assure that $H$ is $\myr$-saturated. In the remaining two cases this aspect is not a problem, but for $\myvi$ we again see that $\myr=H$ is the only choice because of condition (ii) of Definition \ref{defnap}. For $\myvii$, however, both $(H,H)$ and $(H,H\cup\{v\})$ are \naps.
\end{example}

\subsection{Properly infinite projections}\label{SSecPre}

We prepare now some notation about projections. 

\begin{definition}
For two projections $p,q$ in a \Ca $A$, 
we write $p \sim q$ if they are Murray-von Neumann equivalent, 
that is, there exists $v \in A$ satisfying $v^*v=p$ and $vv^*=q$. 
We write $p \precsim q$ 
if there exists a projection $p' \sim p$ with $p' \leq q$. 
\end{definition}

$V(A)$ denotes all equivalence classes $[p]$ of projections in
$A\otimes \bK$ by the relation ``$\sim$''. As usual, $V(A)$ becomes a semigroup (in fact, an abelian monoid) by letting the sum of $[p]$ and $[q]$ equal $[p\oplus q]$.
We  write $[p]\leq [q]$ when
$p\precsim q$.

\begin{definition}
A projection $p$ is said to be 
\begin{enumerate}[(i)]
\item {\em finite}
if $p' \sim p$ and $p' \leq p$ implies $p'=p$. 
\item {\em infinite} if there exists a projection $p' \sim p$ 
satisfying $p' \leq p$ and $p' \neq p$. 
\item {\em stably finite}
if $p \otimes 1 \in A \otimes M_n(\C)$ is finite 
for all $n=1,2,\ldots$
\item {\em stably infinite} 
if $p \otimes 1 \in A \otimes M_n(\C)$ is infinite 
for some $n=1,2,\ldots$. 
\end{enumerate}
\end{definition}

\begin{definition}
A \Ca $A$ is said to be 
\begin{enumerate}[(i)]
\item {\em finite} if all projections of $A$ are finite, 
\item {\em infinite} if $A$ contains an infinite projection, 
\item {\em stably finite} if $A \otimes \bK$ is finite, 
\item {\em stably infinite} if $A \otimes \bK$ is infinite, 
\end{enumerate}
where $\bK$ is the \Ca of all compact operators 
on the separable infinite-dimensional Hilbert space. 
\end{definition}

The following is well-known and easy to see.

\begin{lemma}\label{passtocornerciteme}
A projection $p$ in a \Ca $A$ is finite 
(resp.\ infinite, stably finite, stably infinite) 
if and only if the corner $pAp$ is finite 
(resp.\ infinite, stably finite, stably infinite). 
\end{lemma}

\begin{definition}\label{pidef}
A projection $p \in A$ is said to be {\em properly infinite} 
if there exist two mutually orthogonal projections $p'$ and $p''$ 
such that $p' \sim p'' \sim p$ and $p' + p'' \leq p$. 
\end{definition}

Is semigroup language, the property translates to 
\[
2[p]\leq [p]
\]
Of course then $n[p]\leq [p]$ for any $n$. We note that for any \Ca $A$, 
$0 \in A$ is the only projection which is finite and properly infinite. 
 Obviously, the sum of two orthogonal properly infinite  projections is
again properly infinite. 
\subsection{Semiprojectivity}%\infset

A $C^*$-algebra $A$ is said to be \textbf{semiprojective} if for any given 
map
$\phi:A\to D/\overline{\cup J_k}$, where
\[
J_1\lhd J_2\lhd J_3\lhd\cdots
\]
are nested ideals of $D$, there exists a partial lift $\psi:A\to D/J_{k}$ for some $k$:
\[
\xymatrix{
&D/J_k\ar[d]\\
A\ar_-{\phi}[r]\ar@{..>}@/^/[ur]^-{\psi}&D/\overline{\cup J_k}.}
\]
We call $A$ \textbf{weakly semiprojective w.r.t.} a certain class $\mathcal C$ of $C^*$-algebras when similarly
\[
\xymatrix{
&\prod_{i=1}^\infty B_i\ar[d]\\
A\ar_-{\phi}[r]\ar@{..>}@/^/[ur]^-{\psi}&\prod_{i=1}^\infty B_i\left/ \bigoplus_{i=1}^\infty B_i\right. }
\]
whenever $B_i\in \mathcal C$. When $\mathcal C$ is the class of all $C^*$-algebra, we just say that $A$ is weakly semiprojective.
It is easy to see that a semiprojective $C^*$-algebra is also weakly semiprojective. See \cite{setal:ccsr} for further discussion of this property.

We prepare the following notion: 

\begin{definition}
An inclusion of \CA s $A_0 \subseteq A$ 
is said to have \emph{weak relative lifting property} 
if for every \CA s $B,D$, 
every surjective \shom $\pi\colon D \to B$, 
every \shom $\varphi\colon A \to B$, 
and every \shom $\psi\colon A_0 \to D$ 
such that $\pi \circ \psi = \varphi|_{A_0}$ 
there exists a \shom $\overline{\varphi}\colon A \to D$ 
such that $\pi \circ \overline{\varphi} = \varphi$. 
\end{definition}

We use the term ``weak'' 
since we do not assume that $\overline{\varphi}|_{A_0}=\psi$. Thus in the diagram
\[
\xymatrix{
A_0\ar[r]^-{\psi}\ar@{^{(}->}[d]&D\ar[d]^-{\pi}\\
A\ar@{-->}[ur]_-{\overline{\phi}}\ar[r]_-{\phi}&B,}
\]
the top left triangle does not necessarily commute. The reader is invited to compare with \cite[\S 5.1]{elp:sar}.

The following results follow directly from the definition.

\begin{lemma}\label{weakone}
Let $A_1 \subseteq A_2 \subseteq A_3$ be inclusions of \CA s. 
If both $A_1 \subseteq A_2$ and $A_2 \subseteq A_3$ 
have weak relative lifting property, 
then so does $A_1 \subseteq A_3$. 
\end{lemma}

\begin{lemma}\label{weaktwo}
Let $A_0 \subseteq A$ be inclusions of separable \CA s
having weak relative lifting property. 
If $A_0$ is semiprojective 
then $A$ is semiprojective. 
\end{lemma}

\section{Model projections}

 A key tool in our analysis stems from the computation of the nonstable $K$-theory of graph algebras and related objects, developed by Ara, Moreno and Pardo in \cite{pamamep:nkga}, and we are grateful to Pere Ara for having educated us on the matter. Succintly put, this work allows us to regard any projection $p\in C^*(E,\mys)\otimes \bK$,  up to Murray-von Neumann equivalence, as being on the form
\[
%\pVX=
\bigoplus_{i=1}^n \pv{v_i}{X_i}
\]
where $v_i\in E^0$ and $X_i\subseteq E^1v_i$ is finite, and where
\[
\pv{v}{X}=p_v-\sum_{e\in X}s_es_e^*
\]
The results of \cite{pamamep:nkga} are stated in the generality of row-finite graph $C^*$-algebras, but passing to the generality of general and possibly relative graph algebras is possible by essentially the same proof as in \cite{pamamep:nkga}. This is  explained in \cite{dhmlmmer:nklpa} for graph algebras, so we just need to extend the observations to relative graph algebras.

\subsection{Identifying semigroups}

\begin{definition}
For any graph $E=(E^0,E^1,s,r)$ and $\mys\subseteq E^0\reg$, we let $W(E,\mys)$ be the universal semigroup (i.e., abelian monoid) generated by
\[
\{\g{v,X}\mid v\in E^0, X\subseteq vE^1, X\text{ finite}\}
\]
subject to the relations
\begin{enumerate}[(i)]
\item $\g{v,X}+\g{s(e),\emptyset}=\g{v,X\setminus \{e\}}$ for any $v,X,e$ with $e\in X\subseteq vE^1$
\item $\g{v,vE^1}=0$ for any $v\in \mys$
\end{enumerate}
\end{definition}

We  abbreviate $\g{v,\emptyset}=\g{v}$, and think of (i) as the equivalent 
\begin{equation}\label{twoprime}
\g{v,X\cup\{e\}}+\g{s(e)}=\g{v,X}
\end{equation}
which is valid whenever $e\not\in X$, $X\subseteq vE^1$, $e\in vE^1$.

\begin{lemma}
There exists a homomorphism $\Phi:W(E,\mys)\to V(C^*(E,\mys))$ given by
\[
\Phi(\g{v,X})=\left[p_v-\sum_{e\in X}s_es_e^*\right]
\]
\end{lemma}
\begin{proof}
The element considered in $C^*(E,\mys)$ is really a projection because of \CK{2}, and we have that 
\begin{eqnarray*}
\left[p_v-\sum_{f\in X\backslash \{e\}}s_fs_f^*\right]+[p_{s(e)}]&=&
\left[p_v-\sum_{f\in X}s_fs_f^*-s_es_e^*\right]+[s_e^*s_e]\\
%&=&
%\left[p_v-\sum_{f\in X}s_fs_f^*-s_es_e^*\right]+[s_es_e^*]\\
&=&\left[p_v-\sum_{f\in X}s_fs_f^*\right]
\end{eqnarray*}
by \CK{1} and the definition of Murray-von Neumann equivalence, and
\[
\left[p_v-\sum_{e\in vE^1}s_es_e^*\right]=[0]
\]
by \CK{3}, when $v\in \mys$. Then $\Phi$ exists by the universality of $W(E,\mys)$.
\end{proof}

We have

\begin{theorem}\label{amp}
$\Phi$ is a semigroup isomorphism.
\end{theorem}
\begin{proof}
Our identification of relative graph algebras  as graph algebras give us the vertical semigroup isomorphisms in 
\[
\xymatrix{
{W(E,\mys)}\ar[r]^-{\Phi}&{V(C^*(E,\mys))}\\
{W(\makestandard{E}{\mys})}\ar[r]^-{\Phi}\ar[u]&{V(C^*(\makestandard{E}{\mys}))}\ar[u]}
\]
We then appeal to \cite{dhmlmmer:nklpa}.
\end{proof}

\subsection{Model projections}

We proved in the previous section that 
when $p\in C^*(E,\mys)\otimes \bK$ is given, we have $p\sim \bigoplus_{i=1}^n \pv{v_i}{X_{v_i}}$ for a suitable choice of $v_1,\dots, v_n$ and $X_{v_1},\dots, X_{v_n}$.
We aim to determine when such projections
are infinite,  properly infinite, and when the corners they define are semiprojective $C^*$-algebras. For these purposes it is very convenient when there are no multiplicities among the elements $v_i$, since then we may set
\[
X=\bigcup_{i=1}^n X_{v_i}
\]
without loss of information and consider
\[
\pVX=\sum_{v\in V}\pv{v}{X_v}\in C^*(E,\mys).
\]
instead of the direct sum, these being the same up to Murray-von Neumann equivalence. Of course, not every projection in $C^*(E,\mys)\otimes\bK$ has a representation like that, but as we shall see, for our purposes we may still reduce to this case.

\begin{lemma}\label{stdform}
For any projection $p\in C^*(E,\mys)\otimes \bK$ there exists a finite set $V\subseteq E^0$ and finite sets $\{X_v\}_{v\in V}$ with $X_v\subseteq vE^1$ 
and
\begin{equation}\label{nobull}
X_v=vE^1\text{ only when }v\not \in \mys
\end{equation}
so that with $\pVX$ defined as above, we have $[\pVX]\leq [p]\leq n[\pVX]$ for some $n$.
\end{lemma}
\begin{proof}
  Using Theorem \ref{amp}, we see that $p\sim \bigoplus_{i=1}^n \pv{v_i}{X_i}$ for a suitable choice of $v_1,\dots v_n$ and $X_1,\dots X_n$. We then note that 
\[
p_{v,X}\oplus p_{v,Y}\sim p_{v,X\cap Y}\oplus p_{v,X\cup Y}
\]
by redistributing, to the first summand,  the terms $s_es_e^*$ which occur in both terms. Note that we have $p_{v,X\cup Y}\leq p_{v,X\cap Y}$ and $p_{v,X\cap Y}\in I(p_{v,X\cup Y})$. Thus, since for any projection $q$ we get
\[
[p_{v,X\cup Y}\oplus q]\leq[p_{v,X\cap Y}\oplus p_{v,X\cup Y}\oplus q]\leq [p_{v,X\cap Y}]+m[\oplus p_{v,X\cap Y}]+ [q]\leq (m+1)[p_{v,X\cup Y}\oplus q]
\]
the proof can be completed
by
reordering and applying this argument inductively.

If $v\in \mys$ and $X_v=vE_1$ we have $\pv{v}{X_v}=0$ by \CK{3}, so we may simply drop such a summand.
\end{proof}

We will denote such projections \textbf{model projections} and the pair $(V,\XV)$ the \textbf{model}. Whenever such a model is given, it is understood that $X\subseteq VE^1$, and we use the notation
\[
X_v=X\cap vE^1
\]
whenever convenient.

In general, a corner $\pVX C^*(E,\mys)\pVX$ does not have an obvious  representation as a graph algebra, but we will now devise a procedure of replacing a model $(V,X_v)$ by another model $(W,Y_w)$ -- for a different projection, but retaining all pertinent information for our purposes --  so that we get $\pWY C^*(E,\mys)\pWY$ that does. Envisioning potential future applications to the case of non-unital graph $C^*$-algebras, we allow in the present section $V$ and $W$ to be infinite, whereas all sets $X_v$ and $Y_w$ must be finite. In this case, the projections $\pVX$ and $\pWY$ are to be considered as elements of the multiplier algebra  $M(C^*(E,\mys))$. For all purposes of the paper at hand, except Example \ref{bbwasagainnotentirelyright}, the reader may wish to assert that $V$ is finite and $\pVX$ is in $C^*(E,\mys)$.

\begin{proposition}\label{cornerisgraph}
Let $(V,\XV)$ be a model and set
\[
H=\{v\in V\mid X_v=\emptyset\}.
\]
When
\begin{enumerate}[(i)]
\item $H$ is hereditary 
\item $X_v=vE^1(E^0\setminus H)$ for all $v\in V\setminus H$
\end{enumerate}
then for the graph $(F^0,F^1,s,r)$ defined by
\[
F^0=V\qquad F^1=HE^1\sqcup (V\setminus H)E^1H
\]
and $s,r$ chosen as restrictions from $E$, we have
\[
\pVX C^*(E,\mys)\pVX\simeq C^*(F,\mys\cap H).
\]
\end{proposition}

The idea of the definition is that $X_v=\emptyset$ precisely at some hereditary set $H$, and that at other vertices in $V$, $X_v$ contains exactly those edges that start outside of $H$. Note that $\mys$ only plays a indirect role here -- the places where \CK{3} are imposed are simply passed on.

\begin{proof}
Let us abbreviate $P=\pVX$.We provide a concrete isomorphism $\phi:C^*(F,\mys\cap H)\to P C^*(E,\mys)P$ given by
\[
\phi(p_v)=q_v\qquad \phi(s_e)=t_e
\]
where
\begin{eqnarray*}
q_v&=&p_v-\sum_{e\in X_v} s_es_e^*=\pv{v}{X_v},\qquad v\in F^0=V\\
t_e&=&s_e, \qquad e\in F^1=HE^1\sqcup (V\setminus H)E^1H
\end{eqnarray*}
To see that $\phi$ is well-defined we must check that $Pq_vP=q_v$, $Pt_eP=t_e$, and that the axioms \CK{1}--\CK{3} hold true. For $v\in V$, we obviously have $\pv{v}{X_v}\leq P$$e\in F^1$, we have $s_e^*s_e=p_{s(e)}=q_{s(e)}$ since $s(e)\in H$ as $H$ is hereditary. Also, we get that
\[
s_es_e^*\leq q_{r(e)}
\]
This is clear for $e\in HE^1$, and  for $e\in (V\setminus H)E^1H$ follows by noting that we excluded all edges there from $X_v$. Thus the initial projection of $t_e$ is dominated by $P$, and this is the only non-trivial point in checking that \CK{1} and \CK{2} holds, and that $\phi$ maps to the corner $PC^*(E,\mys)P$. To check \CK{3}, suppose $v\in \mys\cap H$ is given, then
\[
\sum_{e\in vF^1}t_et_e^*=\sum_{e\in vE^1}s_es_e^*=p_v=q_v
\]
since $H$ is hereditary.

The map is injective by gauge invariance, cf.\ Theorem \ref{giut}. Indeed, $q_v\not =0$ for all $v\in H$ since $q_v=p_v\not=0$,  and for the remaining vertices since we arranged that $X_v=vE^1$ only for $v\not\in \mys$.  
Suppose $v$ and $X\subseteq vF^1$ is given with
\[
q_v-\sum_{e\in X}t_et_e^*=p_v-\sum_{e\in X_v\sqcup X}s_es_e^*=0.
\]
By universality in $C^*(E,\mys)$ we get that $v\in \mys$ and $X_v\sqcup X=vE^1$. If $v\in \mys\setminus H$ there is nothing to prove, and if $v\in\mys\cap H$ we have $X_v=\emptyset$, whence $X=vE^1=vF^1$ as desired, using that $H$ is hereditary. 

Finally,  $\phi$ is surjective by the following considerations. Note that for any path $\xi$ in $E$ with $|\xi|\geq 1$, we have
\[
P s_\xi=\begin{cases} s_\xi&r(\xi)\in H\text{ or } \xi_1\in (V\setminus H)E^1H\\0&\text{otherwise}
\end{cases}
\]
whereas for $v\in E$, we get
\[
P=\begin{cases} q_v&v\in H\\0&\text{otherwise}.
\end{cases}
\]
Hence we have
\begin{eqnarray*}
PC^*(E,\mys)P
&=&P\overline{\operatorname{span}}(s_\xi s_\eta^*\mid \xi,\eta\in E^*,|\xi|,|\eta|\geq 0\})P\\
&=&\overline{\operatorname{span}}(P s_\xi (P s_\eta)^* \mid |\xi|,|\eta|\geq 0\})\\
&=&\overline{\operatorname{span}}\left(\{s_\xi s_\eta^*\mid \xi,\eta\in F^*,|\xi|+|\eta|\geq 1\}\cup\{q_v\mid v\in F^0\}\right)
\end{eqnarray*}
\end{proof}

\begin{proposition}\label{passto}
For any model $(W,\YW)$ there exists a model $(V,\XV)$ such that 
\begin{enumerate}[(i)]
\item The pair $(V,\XV)$ satisfies the conditions of Proposition \ref{cornerisgraph}
\item $[\pWY]\leq [\pVX]\leq n[\pWY]$ for some $n$.
\end{enumerate}
Here $V$ is given as $W\cup H$, where $H=\bigcup_{n=0}^\infty H_n$ with
\begin{gather*}
H_0=\{w\in W\mid Y_w=\emptyset\}\cup\bigcup_{\{w\in W\mid Y_w\not=\emptyset\}} s(wE^1\backslash Y_w)\\
H_{n+1}=H_n\cup\{w\in W\mid Y_w\subseteq H_{n}E^1\}\cup\{v\in E^0\mid H_nE^1v\not=\emptyset\}
\end{gather*}
and
 the $X_v$ are given by
\[
X_v=\begin{cases}
\emptyset &v\in H\\
vE^1(E^0\setminus H)& v\in W\backslash H
\end{cases}
\]
\end{proposition}
\begin{proof}
Because of the last subset in the definition of $H_{n+1}$, we clearly get that $H$ is hereditary. 
We will check
\begin{enumerate}
\item $Y_w\setminus E^1H=wE^1(E^0\setminus H)\not=\emptyset$ for any $w\in W\setminus H$
\item $\pv{v}{X_v}\in\ideal{\pWY}$ for any $v\in V$
\end{enumerate}
which then entail the claims; indeed by (1) $(X,\XV)$ will be a model because $X_v$ is finite for any $v$, being either empty or contained in $Y_v$, and the extra requirements of Proposition \ref{cornerisgraph} are met by construction. Further, we have $\pWY\leq \pVX$ by (1) again, and $[\pVX]\leq n[\pWY]$ for some $n$ by (2).

To prove (1), we first note that for any $w\in W\backslash H$ we have $Y_w\not=\emptyset$ by definition of $H_0$, and $wE^1(E^0\setminus H)\not=\emptyset$ because of the middle subset in the recursive definition of $H_{n+1}$. We obviously have the inclusion from left to right, so assume that $e\in wE^1(E^0\setminus H)$. To see that $e\in Y_w$ we note that if 	$e\not\in Y_w$, we would have $s(e)\in H$, which would be a contradiction. To prove (2), we set $I=\ideal{\pWY}$ and first prove by induction that when $v\in H_n$, we have $p_v\in I$. When $n=0$, this is clear for $v\in V$ with $X_v=\emptyset$, and for $v\in s(wE^1\backslash Y_w)$ with $Y_w\not=\emptyset$ we have for a suitable $e\in wE^1\setminus Y_w$ that
\[
p_v=s^*_es_e\sim s_es_e^*\leq p_w-\sum_{f\in Y_w}s_fs_f^*\in I
\]
For any $n$, assume that the claim is proved for $H_n$ and fix $v\in H_{n+1}$. When $v\in H_n$ or when $H_nE^1v\not=\emptyset$, we easily get that $p_v\in I$, and when $v$ was added to $H_n$ because $Y_v\subseteq H_nE^1$ we have
\[
p_v=\left(p_v-\sum_{f\in Y_v}s_fs_f^*\right)+\sum_{f\in Y_v}s_fs_f^*
\]
where the first summand is dominated by $\pWY$ by construction, and each summand in the last term lies in $I$ by the induction hypothesis, since for each such $f$ we have $s(f)\in H_n$.
\end{proof}

When $E^0$ is finite there is an obvious algorithm computing this model, running until $H_n=H_{n+1}$. We denote this algorithm \textsc{HModel}, so in the theorem above we have $(V,\XV)=$\textsc{HModel}$(W,\YW)$.
\section{Dichotomies of infinity}
In the present  section we will prove the  two dichotomies stated in Theorem \ref{dicho} regarding projections in relative graph algebras or their stabilizations: either a projection is infinite, or any surjective image of it is stably finite, and either a projection is properly infinite, or some surjective image of it is nonzero and stably finite.
Our strategy will be to prove these results by reducing the first claim to the case when $p$ is a model projection, and then applying a version of the dichotomy in \cite{ekmr:npic}, somewhat amended to this setting, to derive the second. We obtain this by giving a complete algorithmic description of when a model projection $\pVX$ is infinite, but postpone to Section \ref{SecPI} the much more difficult question of devising an algorithm to decide which model projections are  properly infinite.

\subsection{Finite corners}

We now describe a straightforward algorithm to check whether or not a model projection is infinite.  To
run it, and several other algorithms provided below, we assert that the
vertex set $E^0$ is partitioned into three sets 
\[
E^{0}_0,
E^{0}_1,
E^{0}_2
\]
 given by the properties that there is no cycle based on $v\in
E^0_0$, a unique cycle based on $v\in E^0_1$, and more than one cycle
based on $v\in E^0_2$. When $E^0$ is finite, there are $O(|E^0|^3)$
algorithms for determining this, related to Warshall's algorithm (cf. \cite{kbr}), but in general it is of course not possible to decide how to partition an infinite $E^0$ into such sets. All our algorithms given below could be amended to compute only partial information (those vertices in a finite set supporting cycles of bounded length) and terminate as indicated, but we will not pursue this matter here.

We now set 
\[
\infset=E^0_2\cup \{v\in E^0_1 \mid v\not\in\mys\text{ or }|vE^1|>1\},
\]
noting
\begin{lemma}\label{infibasics} Consider a vertex projection $p_v\in C^*(E,\mys)$.
\begin{enumerate}[(i)]
\item When $v\in E^0_2$, $p_v$ is properly infinite. \item When $v\in \infset$, $p_v$ is infinite. \item When $\infset=\emptyset$, then $C^*(E,\mys)$ is an $AT$ algebra.
\end{enumerate}
\end{lemma} 
\begin{proof}
When $\xi_1,\xi_2$ are two different cycles based at  $v$ we have
\[
p_v\sim s_{\xi_1}^*s_{\xi_1}\sim s_{\xi_2}^*s_{\xi_2}
\]
with the two latter projections orthogonal and dominated by $p_v$.
When there is a unique cycle $\xi$ based at $v$, we have in all cases that $p_v\sim s_{\xi}^*s_{\xi}$ and $p_v\geq  s_{\xi}s_{\xi}^*$, and equality in the latter case only when $\xi$ has no entry and all vertices visited by $\xi$ are in $\mys$. Thus we get infinite projections in the remaining cases. Finally, when $|E^0|<\infty$ and $\infset =\emptyset$ we  see that $C^*(E,\mys)$ is a direct sum of algebras of the form $M_n(\C)$ or $M_n(C(\T))$. In the general case we write $E$ as an increasing union of finite subgraphs and note that because of Lemma \ref{subgraph} it suffices to prove that $\infset=\emptyset$ implies that $F^0_{\infty,(F\cap \mys)\cup \EF}=\emptyset$ for any such subgraph $F$.  The only concern is that we could have that a $v\in F^0_1$ had $v\in \mys$ but $v\in \EF$, but in this case we would have to have $v\in E^0_1$ and then see that $v$ receives no edge from outside $F$.
\end{proof}

\begin{theorem}\label{Thm:notI}
Let $E=(E^0,E^1,s,r)$ be a graph, 
and $\mys$ be a subset of $E^0\reg$. 
For a model projection $p_{V,\XV}$,
the following conditions are equivalent. 
\begin{enumerate}[(i)]

\item $\pVX$ is finite in $C^*(E,\mys)$. 
\item for any gauge-invariant ideal $I$ of $C^*(E,\mys)$ 
the image of $\pVX$ in the quotient $C^*(E,\mys)/I$ 
is stably finite. 
\item for any ideal $I$ of $C^*(E,\mys)$ 
the image of $\pVX$ in the quotient $C^*(E,\mys)/I$ 
is 
stably finite. \item $\left[\{v\in V\mid X_v=\emptyset\}\cup\bigcup_{\{v\in V\mid X_v\not=\emptyset\}}s(vE^1\backslash X_v)\right]E^*\infset=\emptyset$

\end{enumerate}
\end{theorem}

\begin{proof}
Obviously, (iii)$\Longrightarrow$(ii)$\Longrightarrow$(i), and we have seen in Lemma \ref{infibasics} that all $p_v$ with $v\in \infset$  are infinite, so since (iv) implies that $\pVX$ dominates a projection equivalent to one of these vertex projections, we also have that (i)$\Longrightarrow$(iv).  For the remaining implication, note that we can pass from $\pVX$ to $\pWY$ with the latter satisfying the condition of Proposition \ref{cornerisgraph} since by Lemma \ref{passto}(ii), if $\pWY$ is stably finite, then so is $\pVX$. We have that $\pWY$ defines  a corner isomorphic to $C^*(F,\mys_F)$ which will have $\F^0_{\infty,\mys_F}=\emptyset$ if only $H\cap \infset=\emptyset$ with $H$ as in the proof of Proposition \ref{passto}. It is clear that when $v\in H_0$ or when $v\in H_{n+1}\setminus H_n$ is added because $H_nE^1v\not=\emptyset$, then $v\not\in \infset$ by our assumption in (iv), so consider the case when $v\in H_{n+1}\setminus H_n$ is added because $X_v\in E^1H_n$. In that case we have $vE^1\subseteq H_n$, and since we may assume by induction that no path from $\infset$ ends in a vertex of $H_n$, the same can be said of $v$. 

We conclude that $C^*(F,\mys_F)$ is an  AT algebra as seen in Lemma \ref{infibasics}(iii), and hence has the property that every surjective image of it is stably finite, proving (iv)$\Longrightarrow$(iii). 
\end{proof}

\begin{proposition}\label{dichoI}
Let $p\in C^*(E,\mys)$. Either $p$ is infinite, or for any surjection $\pi:C^*(E,\mys)\to A$, $\pi(p)$ is stably finite.
\end{proposition}
\begin{proof}
Given $p$, we choose a model $\pVX$ with 
\[
[\pVX]\leq [p]\leq n[\pVX]
\]
If $[p]$ is finite, then so is $[\pVX]$, and by Theorem \ref{Thm:notI}(ii) we have that $m[\pVX+I]$ is finite for all $m$ and $I\idealof C^*(E,\mys)$. Thus
$m[p+I]$
is finite for any $m$, being dominated by $mn[\pVX+I]$.
\end{proof}

It is not true in general that any infinite projection in $A/I$ which lifts to a projection also lifts to one which is infinite (an easy  counterexample is found in $C\mathcal O_2^\sim\to \mathcal O_2$). Further, by \cite{mr:scfip}, it is also not true in general that any sum of finite projections is finite. But in our case, such results may be inferred from our results above:

\begin{corollary}
Let $I\gidealof C^*(E,\mys)$ be given. When $p'\in C^*(E,\mys)/I$ is an infinite projection, there exists an projection $p\in C^*(E,\mys)$ such that $p'=p+I$, and any such lift is infinite.
\end{corollary}
\begin{proof} 
We know every image of an finite projection is finite by Proposition \ref{dichoI}, so such a lift must be infinite.
\end{proof}

\begin{corollary}\label{Inicesum} Let $p,q\in C^*(E,\mys)\otimes\bK$ be given. When $p\oplus q$ is infinite, then so is either $p$ or $q$.
\end{corollary}
\begin{proof}
We may assume that $p$ and $q$ are given by models $\pVX$ and $\pWY$, respectively, and then as seen in Lemma \ref{stdform} we have that $p\oplus q$ is given by the model
\[
(V\cup W,X\cup Y)
\]
with
\[
(X\cup Y)\cap vE^1=X'_v\cap Y'_v
\]
where we set $X'_v=vE^1$ for every $v\not\in V$ and $X'_v=X_v$ otherwise, and similarly for $Y$. By (iv) of Theorem \ref{Thm:notI} we know that there must be a path starting at $\infset$ which terminates at  some $v\in V\cup W$ in such a way that the last edge of this path is not in $X_v'\cap Y_v'$. In the case that $v\not\in V$, we have that $X_v'\cap Y_v'=Y_v$, and hence we conclude that $q$ is infinite. Similarly, $p$ must be infinite when $v\not \in W$, and when $v\in V\cap W$ we get that the given path has its last edge $e$ outside of $X_v'\cap Y_v'=X_v\cap Y_v$. When $e\not \in X_v$ we may conclude that $p$ is infinite, and when $e\not \in Y_v$ that $q$ is.
\end{proof}

We end this section by giving a concrete algorithm for checking infinity of model projections:

\begin{algorithm}
\caption{\textsc{InfiniteModel}$(V,X)$}
\begin{algorithmic}[1]
\REQUIRE $V=\{v_1,\dots, v_n\}$ a finite subset of $E^0$, $X=\bigsqcup_{i=1}^nX_{v_i}$, $X_{v_i}\subseteq v_iE^1$ finite.
\STATE $W\leftarrow \left[\{v\in V\mid X_v=\emptyset\}\cup\bigcup_{X_v\not=\emptyset}s(vE^1\backslash X_v)\right]$
\STATE $\infset\leftarrow E^0_2\cup \{v\in E^0_1 \mid v\not\in\mys\text{ or }|vE^1|>1\}$
\STATE $D\leftarrow \emptyset$
\WHILE{$W\subseteq D$}
\IF{$W\cap \infset \not=\emptyset$}
\RETURN $\infty$
\ENDIF
\STATE $D\leftarrow D\cup W$
\STATE $W\leftarrow r^{-1}(s(W))$
\ENDWHILE
\RETURN $1$
\end{algorithmic}
\end{algorithm}

\begin{proposition}\label{Prop:flowchartI}
Let $C^*(E,\mys)$ be a graph $C^*$-algebra, and $(V,\XV)$ a model.
\begin{enumerate}[(i)]
\item If \textsc{InfiniteModel}$(V,X)$ terminates with $\infty$, then $\pVX$ is 
  infinite
\item If \textsc{InfiniteModel}$(V,X)$ terminates with $1$, then $\pVX$ is 
  finite
\item If \textsc{InfiniteModel}$(V,X)$ does not terminate, then $\pVX$ is 
  finite
\end{enumerate}
When $E^0$ is finite, the procedure always terminates.
\end{proposition}
\begin{proof}
When $E^0$ is finite, the procedure must terminate because the number of elements in $D$ increases by at least one in any iteration. The remaining claims are straightforward by Theorem \ref{Thm:notI}(iv).
\end{proof}

\subsection{Properly infinite projections}

\begin{theorem}\label{Thm:notPI:pre}
Let $E=(E^0,E^1,s,r)$ be a graph, 
and $\mys$ be a subset of $E^0\reg$. 
For any projection $p\in C^*(E,\mys)$, 
the following conditions are equivalent. 
\begin{enumerate}[(i)]
\item  $p$ is not 
properly infinite in $C^*(E,\mys)$. 
\item there exists a gauge-invariant ideal $I\gidealof C^*(E,\mys)$ 
such that the image of $p$ in the quotient $C^*(E,\mys)/I$ 
is non-zero and stably finite. 
\item there exists an ideal $I\idealof C^*(E,\mys)$ 
such that the image of $p$ in the quotient $C^*(E,\mys)/I$ 
is non-zero and stably finite. 
\end{enumerate}\end{theorem}
\begin{proof}
Consider 
\begin{enumerate}[(i')]\addtocounter{enumi}{1}
\item there exists a gauge-invariant ideal $I\gidealof C^*(E,\mys)$ 
such that the image of $p$ in the quotient $C^*(E,\mys)/I$ 
is non-zero and finite. 
\item there exists an ideal $I\idealof C^*(E,\mys)$ 
such that the image of $p$ in the quotient $C^*(E,\mys)/I$ 
is non-zero and finite. 
\end{enumerate}
We have that (i) and (iii') are equivalent by \cite{ekmr:npic}.
Clearly
\[
\text{(ii')}\Longrightarrow \text{(iii')}\Longleftarrow \text{(iii)}\Longleftarrow\text{(ii)}
\]
and (ii') implies (ii) by Theorem \ref{dicho}(i). We close the circle by seeing that (iii') implies (ii') as follows. First choose a model projection according to Proposition \ref{stdform} so that
\[
[\pVX]\leq [p]\leq n[\pVX]
\]
Now let $I\idealof C^*(E,\mys)$ be such that $p+I$ is finite, and note that the same is true for $\pVX+I$.
Define $J\subseteq I$ with $J\gidealof C^*(E,\mys)$ as
\[
J=\bigcap_{z\in \T}\gamma_z(I)
\]
We first aim to prove that $\pVX+J$ is finite as well. For this, consider
\[
\phi:C^*(E,\mys)/J\to C(\T,C^*(E,\mys)/I)
\]
defined by
\[
\phi(x+J)(z)=\gamma_z(x)+I.
\]
We see that $\phi$ is well-defined since $\gamma_z$ is point-norm continuous and by the definition of $J$, indeed if $x\in J$ we have $\gamma_z(x)\in I$ since $x\in J\subseteq \gamma_{\overline z}(I)$. Similarly, $\phi$ is injective, Now note that $\phi(\pVX)$ is finite in $C(\T,C^*(E,\mys))$. Indeed, if it were not, we would have $\phi(\pVX )(z)$ infinite  at some $z$, which would contradict that
\[
\phi(\pVX)(z)=\gamma_z(\pVX)+I=\pVX+I.
\]
Consequently, $\pVX$ is finite, and hence so is $p$.
\end{proof} 

\begin{remark}
In \cite{mr:sfp}, Mikael R\o rdam studied 7 different notions of finiteness of projections. For graph algebras, the results above reduce these notions to three, namely, in the notation of \cite{mr:sfp},
\[
\left\{\begin{array}{c}STF\\TF\end{array}\right\}\Longrightarrow
\left\{\begin{array}{c}F\\SF\end{array}\right\}\Longrightarrow
\left\{\begin{array}{c}NPI\\SNPI\\WTF\\SWTF\end{array}\right\}
\]
Graph algebra examples that the arrows may not be reversed are given in \cite{mr:sfp}: the units of the Toeplitz algebra and $\bK^{\sim}$, and in fact (which we shall not pursue here) only finite-dimensional projections in a graph $C^*$-algebra have the property of being tracially finite. Since graph algebras are known to be corona factorizable, the fact that being properly infinite is a stable property confirms a conjecture which is already known in the real rank zero case.
\end{remark}
 
 We are now ready to formulate a result which dramatically simplifies our work with proper infiniteness in graph algebras:
 
 \begin{corollary}\label{powerc}
 Let $p,q,r$ be projections in a graph $C^*$-algebra $C^*(E,\mys)$.
 \begin{enumerate}[(i)]
 \item When $p,q\perp r$ and, for some $n$, $[p]\leq [q]\leq n[p]$, we have
 \[
 p+r\text{ properly infinite}\Longleftrightarrow  q+r\text{ properly infinite}
 \]
 \item When $p\perp q$ and, for some $n$, $[q]\leq n[p]$, we have
\[
 p\text{ properly infinite}\Longleftrightarrow  p+q\text{ properly infinite}
 \]
 \end{enumerate}
 \end{corollary}
 \begin{proof}
To prove (i), we first assume that $2[p+r]\leq [p+r]$ and see that
\[
2[q+r] \leq 2n[p]+2n[r]  \leq [p+r]\leq [q+r].
\]
In the other direction, when $2[q+r]\leq [q+r]$, we have 
\[
2n[p+r] \leq 2n[q+r] \leq [q+r] \leq n[p+r]. 
\]
which shows that  $n[p]$ is properly infinite. This suffices as seen above. For (ii), we let $r=0$ and $q=p+q$ in (i).
\end{proof}
 
\section{Property (FQ)}\label{SecPI}

In the previous section we saw in Corollary \ref{Inicesum}  that determining whether or not 
\[
p_V=p_{\{v_1,\dots,v_n\}}=p_{v_1}+\cdots+p_{v_n}
\]
is infinite reduces to checking whether or not this is the case for one of the constituent projections $p_{v_i}$. Turning now fully to the case of properly infinite projections, we are faced with the problem that 
there is in general no way to examine the proper infiniteness of $p_V$
by studying the proper infiniteness of $p_{v_i}$. Indeed, Example \ref{sumexxs} below exhibits examples in a graph algebra showing that the only thing that can be said in this setting is that the sum of two properly infinite projections is itself properly infinite.

Instead, we will develop the condition (FQ) below, a complete test for proper infiniteness, and show how to check this condition by a concrete algorithm,  guaranteed to terminate when the number of vertices in the graph is finite. We  define 4 properties (TQ), (MQ), (AQ) and (FQ) as follows:

\begin{definition} For $V \subseteq E^0$ we say that
\begin{enumerate}[(i)]
\item $V$ has \emph{Property (TQ)} for a vertex $w \in E^0$
if there exists a first returned cycle $\ell$ on $w$
such that
\[
0<\big|\{\xi \in E^*\mid s(\xi)=w, r(\xi) \in V, \xi \notin E^* \ell\}\big|
<\infty
\]
\item 
$V$ has \emph{Property (MQ)} for a vertex $w \in E^0$
if
\[
0<\big|\{\xi \in E^*\mid s(\xi)=w, r(\xi) \in V\}\big|
<\infty
\]
\item 
$V$ has \emph{Property (AQ)} for
a sequence of vertices $\{w_n\}_{n=1}^\infty$ in $E^0$
if
$V$ has Property (MQ) for $w_{n}$
and $w_{n}E^\dagger w_{n+1}\not=\emptyset$ for every $n$.
\item
$V$ has \emph{Property (FQ) relative to $\mys \subseteq E^0$}
if it satisfies either property (TQ) for some $w \in E^0$,
property (MQ) for some $w \not \in \mys$ or
property (AQ) for some sequence $w_n\in E^0$.
\end{enumerate}
\end{definition}
When $V$ has property (FQ) relative to $E^0\reg$ we just say that it has property (FQ).

\begin{remark}\label{FQremark} 
As we shall explain shortly, the letters T, M, A and F in the names of these properties are supposed to the remind the reader of torus algebras (i.e., $M_n(C(\T))$),  matrix algebras, $AF$-algebras, and finite algebras, respectively. Typical instances of these properties are indicated in Figure \ref{fqfig} in the introduction.
Our definition does not rule out that some of the vertices depicted in the (TQ) or (AQ) cases are outside of $\mys$, but in this case we also have an instance of the (MQ) condition, so 
without loss of generality this may always be assumed.

In the situation of Property (TQ), one notes that 
$\ell$ is the unique first returned cycle on $w$.
Hence if $E$ satisfies Condition (K),
then no $V$ satisfies Property (TQ).

In the definition of Property (AQ),
we have required that there exist  paths $\xi_n$
of positive length such that $r(\xi_n)=w_{n+1}$ and $s(\xi_n)=w_{n}$.
Since there is no returned path on $w_n$ according to the definition
of property (MQ),
we see that $w_n \neq w_m$ for $n \neq m$.
Hence if $|E^0| < \infty$
then no $V$ satisfies Property (AQ).
\end{remark}

\begin{proposition}\label{prop:stablyfinite}
Let $E=(E^0,E^1,s,r)$ be a graph with $V\subseteq E^0$ finite and $\mys \subseteq
E^0\reg$. If $V$ has {(FQ) relative to $\mys$}, there exists $I\gidealof C^*(E,\mys)$ such that
$q(C^*(E,\mys)/I)q$ is either $M_n(C(\T))$ or an AF algebra,
where $q =p_V+I \in C^*(E,\mys)/I$.
 \end{proposition}
\begin{proof}
Assume first that $V$ has (TQ) for the vertex $w$. We choose the ideal $I$ given by the \nap
\[
(E^0\setminus r(E^\dagger w),E\reg^0),
\]
noting that the complement of $r(E^\dagger w)$ is in fact hereditary and (globally) saturated. The corresponding corner of the quotient is the graph algebra of the graph traced by all paths between $w$ and $V$, which in this case is finite and generated by a cycle with no entries, so that this corner is isomorphic to $M_n(C(\T))$.

Now consider the case where $V$ has (MQ) for the vertex $w\not\in \mys$. We choose the ideal $I$ given by the \nap
\[
(E^0\setminus r(E^* w),E\reg^0\setminus\{w\}).
\]
Again we easily see that the first set is hereditary, but this time we note that $w$ emits all of its edges into the complement of $r(E^*w)\cup\{w\}$ without being a member of the set, so we need to ensure that $w$ is not a member of the second set. Since $w\not\in\mys$, this can be arranged. Again the corner is a graph algebra of the graph traced by all paths between $w$ and $V$, which in this case is finite and has no cycles, so that this corner is isomorphic to $M_n(\C)$.
In the last case, we take
\[
\left(E^0\setminus \bigcup _{n\in \N}r(E^\dagger w_n),E\reg^0\right),
\]
and get that the corner is given by an infinite graph with no cycles. Hence, the graph algebra is AF.
\end{proof}

We have seen that $p_V$ fails to be properly infinite when $V$ has (FQ). To prove the opposite,  we address the problem of how to algorithmically determine
whether  a given set $V$ has (FQ). We do this by presenting a sequence of algorithms which, taken together, determine whether or not a finite set $V$ has property (FQ) or not in a graph with finite $E^0$, and in general runs infinitely only on those sets $V$ that do have the property. 

Our first algorithm extracts the ``top'' of a given set $V$, a subset $W$ of vertices which cover $V$ in the sense that $W\subseteq s(VE^*)$, and which is chosen as economically as possible. In general, the top of a set is not unique, and it is of no consequence which one is taken, but to well-define our algorithms we provide a concrete algorithm for extracting one such set, which we then work with throughout.

\begin{algorithm}
\caption{\textsc{GetTop}$(V)$}
\begin{algorithmic}[1]
\REQUIRE $V=\{v_1,\dots v_n\}$ a finite subset of $E^0$
\STATE $W\leftarrow \emptyset$
\FOR{ $v\in V$}
\IF{$v\not\in s( WE^*)$}
\FOR{ $w\in W$}
\IF{$w\in s(vE^*)$}
\STATE{$W\leftarrow W\setminus \{w\}$}
\ENDIF 
\ENDFOR
\STATE{$W\leftarrow W\cup \{v\}$}
\ENDIF
\ENDFOR
\RETURN $W$
\end{algorithmic}
\end{algorithm}

%Note that  $\textsc{GetTop}(\emptyset)=\emptyset$.

\begin{lemma}\label{gettop}
Let $V\subseteq E^0$ be a finite set, and let $W=$\textsc{GetTop}$(V)$. We then have
\begin{enumerate}[(i)]
\item $W\subseteq V$ and $W\subseteq s(VE^*)$
\item For any $w_1,w_2\in W$, if $w_1\in s(w_2 E^*)$ then $w_1=w_2$
%\item For any $w\in W$ and $v\in V$, if $v\apath w$ then $w\apath v$ NEED?
\item For some $N$, we have $[p_W]\leq [p_V]\leq N[p_W]$
\end{enumerate}
\end{lemma}
 \begin{proof}
The first two claims follows from inspection of the algorithm, and (iii) follows by the obvious fact that $p_W\leq p_V$ and by (i) since we may infer that $p_V\in I(p_W)$.
\end{proof}

%Note that by (iv) combined with Lemma \ref{sumPI}, if $p_W$ is properly infinite, so is $p_V$, and if $p_W$ is not properly infinite, %there is some $n$ with $p_V\otimes 1\in M_n(C^*(E,\mys))$ not properly infinite.

 We now come to the key algorithm which helps us identify whether a finite set $V$ has $(FQ)$ relative to $\mys$.

\begin{algorithm}
\caption{\textsc{FindFQ}$(V,R)$}
\begin{algorithmic}[1]
\REQUIRE $V,R$ finite subsets of $E^0$ with $R\subseteq E^0_2$
\STATE $W\leftarrow $\textsc{GetTop}$(V\setminus s({R}E^*))$
\IF{$W\cap [E^0_1\cup (E^0_0\setminus \mys)]\not=\emptyset $ }
\RETURN $(\star,\star)$
\ELSE
\RETURN $(s((W\setminus E^0_2)E^1),R\cup(V\cap E^0_2))$
\ENDIF 
\end{algorithmic}
\end{algorithm}

Note that when line 5 is reached, $s((W\setminus E^0_2)E^1$ and $R\cup(V\cap E^0_2)$ will both be finite sets with the latter contained in $E^0_2$. Hence we may iterate $\textsc{FindFQ}$ as long as the output is not $(\star,\star)$. Our key observation is 

\begin{lemma}\label{findFQi} 
When 
\[
(\star,\star)\not =(V',R')=\text{\textsc{FindFQ}}(V,R),
\]
we have
\[
p_{V'\cup R'}\text{ is properly infinite }\Longrightarrow p_{V\cup R} \text{ is properly infinite}
\]
\end{lemma}
\begin{proof}
We establish the sequence of implications
\begin{eqnarray*}
p_{V'\cup R'}\text{ properly infinite }&\Longleftrightarrow &p_{W\backslash E^0_2}+p_{R'\setminus V'} \text{ properly infinite }\\
&\Longrightarrow &p_{W\cup R'}\text{ properly infinite }\\
&\Longleftrightarrow &p_{V\setminus s(RE^*)}+p_{R'\setminus W}\text{ properly infinite }\\
&\Longrightarrow &p_{V\setminus s(RE^*)}+p_{R}\text{ properly infinite }\\
&\Longleftrightarrow &p_{V\setminus s(RE^*)}+p_{R} +p_{(V\cap s(RE^*))\setminus R}\text{ properly infinite }\\
\end{eqnarray*}
which proves the desired result since the last projection is precisely $p_{V\cup R}$.

For the first implication, we note that 
\[
p_{W\backslash E^0_2}=\sum_{v\in W\backslash E^0_2}\sum_{e\in vE^1}s_es_e^*
\]
since by our algorithm, $W\backslash E^0_2\subseteq \mys$. The latter projection is equivalent to
 \[
 \bigoplus_{v\in W\backslash E^0_2}\sum_{e\in vE^1}s_e^*s_e= \bigoplus_{v\in W\backslash E^0_2}\sum_{e\in vE^1}p_{r(e)}
 \]
 where, by definition, $r(e)$ will range in $V'$, with each vertex occurring at least once at at most $n$ for some suitable $n$. Hence we may apply Corollary \ref{powerc}(i). The second implication follows by noting that every vertex in 
 \[
(W\cup R')\setminus [(W\backslash E^0_2)\cup(R'\setminus V')]
\]
lies in $E^0_2$, so that the projection which has been added is automatically properly infinite. We then apply Corollary \ref{powerc}(i) again, this time to 
\[
[p_W]\leq [p_{V\setminus s(RE^*)}]\leq n[p_W]
\]
which  follows by Lemma \ref{gettop}(iv), and establish the fourth implication by noting that every vertex in $R\setminus (R'\setminus W)$ is in $E^0_2$. We end by appeling to  Corollary \ref{powerc}(ii), which applies since $p_{(V\cap s({R}E^*))\setminus R}\in I(p_R)$.
\end{proof}

To further analyze the procedure \textsc{FindFQ} we introduce notation. 
Let $\emptyset\not= V_0\subseteq E^0$ and $R_0\subseteq E^0_2$ be finite sets and assume that \textsc{FindFQ} can be iterated $m+1$ times so that the returned pairs
\begin{eqnarray*}
(V_1,R_1)&=&\text{\textsc{FindFQ}}(V_0,R_0)\\
&\vdots&\\
(V_{n-1},R_{n-1})&=&\text{\textsc{FindFQ}}(V_{n-2},R_{n-2})=\text{\textsc{FindFQ}}^{\circ m}(V_0,R_0)
\end{eqnarray*}
are never $(\star,\star)$ nor have $V_i=\emptyset$, but that 
\[
(V_{m+1},R_{m+1})=\text{\textsc{FindFQ}}(V_{m},R_{m})=\text{\textsc{FindFQ}}^{\circ m+1}(V_0,R_0)
\]
is arbitrary. We denote by $W_k$ the subsets of $V_k$ produced in line 2 of the $(k-1)$st iteration of \textsc{FindFQ}; these sets will then be defined up to $k=m$.

\begin{example}\label{FQrun}
Let us show the steps in a complicated analysis performed by \textsc{FindFQ}.
\begin{center}
%\begin{tabular}{c}
\parbox{125pt}{%
$\xymatrix{
&\bullet&\\
\circ\ar[ur]^(0.9){1}^(0.1){2}\ar[r]&\circ\ar@(ru,rd)[]_-{3}\ar@(r,d)[]\ar[u]&\bullet\ar[lu]_(0.1){4}\\
&\bullet\ar@(l,d)[]^-{5}\ar[u]\ar[ur]&\bullet\ar[u]_(0.1){6}\\
&\circ\ar@(r,d)@{=>}_-{7}\ar[ur]&}$}
%\end{tabular}
\begin{tabular}{c||c|c|c}
$m$&$V_m$&$W_m$&$R_m$\\\hline\hline
0&$\{1\}$&$\{1\}$&$\emptyset$\\
1&$\{2,{3},4\}$&$\{{3},4\}$&$\emptyset$\\
2&$\{5,6\}$&$\{6\}$&$\{3\}$\\
3&$\{7\}$&$\{7\}$&$\{3\}$\\
4&$\emptyset$ &$\emptyset$ &$\{3,7\}$
\end{tabular}
\end{center}
\end{example}

\begin{lemma}\label{goldengoal}
Suppose \textsc{FindFQ}$(V_0,R_0)$ can be iterated $m+1$ times as above, and let $w_m\in W_m$ be given
\begin{enumerate}[(i)]
\item If $w_m\in E^0_0$, then $V_0$ has (MQ) for $w_m$
\item If $w_m\in E^0_1$, then $V_0$ has (TQ) for $w_m$
\item If $m>0$, then $w_m\not\in W_0$.
\end{enumerate}
\end{lemma}
\begin{proof}
We prove this by induction on $m$ with arbitrary choices of $(V_0,R_0)$ on which \textsc{FindFQ} can be iterated $m+1$ times. For $m=0$ and given $w_0$ there is of course 
the path of lenght zero to $V_0$. Let any other path $\alpha$ have $s(\alpha)=w_0$ and $v_0=r(\alpha)\in V_0$. Since there is no path to $R_0$ from $w_0$, the same must be true of $v_0$, and hence by Lemma \ref{gettop}(i) there is a path $\beta$ from $v_0$ to some $w_0'\in W_0$. By Lemma \ref{gettop}(ii) we conclude that $w_0=w_0'$. If $w_0\in E^0_0$ we have that in fact $|\beta\alpha|=0$, so there is no such non-trivial path, proving (i). The claim (ii) follows similarly by noting that $\alpha\beta$ must be a finite number of traversals of the unique cycle based on $w_0$.

When $m>0$ we first note that by construction, $w_m=s(e)$ for some edge $e$ with 
\[
w_{m-1}=r(e)\in W_{m-1}\setminus E_2^0=W_{m-1}\cap E^0_0\cap\mys
\]
By (i) of the induction hypothesis there is a path $\alpha$ from $w_{m-1}$ to $v_0\in V_0$, and as above it may be extended by $\beta$ having $w_0=r(\beta)\in W_0$ since we know that $w_m\not\in s(R_0E^*)\subseteq s(R_mE^*)$. This time we further may conclude that $w_0\in E^0_0\cap \mys$ since we know that $w_0\not\in W_0\cap E^0_2\subseteq s(R_mE^*)$.
Indeed, since $w_0\in E^0_2$ is ruled out this is the only remaining possibility because \textsc{FindFQ} could be iterated more than once.

To prove (iii), assume that $w_m\in W_0$. Then by Lemma \ref{gettop}(ii) we would have $w_m=w_0$ and consequently $w_0\in W_0\cap E^0_1$, which we have seen is impossible. For (i), assume that $w_m\in E^0_0$ and note that the last edge $f$ of $\beta\alpha e$ must have $s(f)\in V_1$. By (i) of the induction hypothesis applied to $(V_1,R_1)$ we see that there is  only a finite number of paths starting at $w_m$ and ending in $V_1$. Since $w_{m-1}$ lies in $\mys$ and in partcicular must be a finite receiver, this means that there are also only finitely many possibilities for $\beta\alpha e$. The proof for (ii) is similar.
\end{proof}

\begin{algorithm}
\caption{\textsc{ProperlyInfinite}$(V)$}
\begin{algorithmic}[1]
\REQUIRE $V$ finite subset of $E^0$
\STATE $R\leftarrow \emptyset$
\WHILE{$V\not=\emptyset$}
\STATE $(V,R)\leftarrow$\textsc{FindFQ}$(V,R)$
\IF{$(V,R)=(\star,\star)$}
\RETURN $1$
\ENDIF
\ENDWHILE
\RETURN $\infty$
\end{algorithmic}
\end{algorithm}

\begin{proposition}\label{Prop:flowchart}
Let $E=(E^0,E^1,s,r)$ be a graph, 
and $\mys$ be a subset of $E^0\reg$. 
Take a finite set $V \subseteq E^0$. 
\begin{enumerate}[(i)]
\item If \textsc{ProperlyInfinite}$(V)$ terminates with $\infty$, then $p_V$ is properly
  infinite
\item If \textsc{ProperlyInfinite}$(V)$ terminates with 1, then $V$
  has property (FQ) relative to $\mys$
\item If \textsc{ProperlyInfinite}$(V)$ does not terminate, then $V$
  has property (FQ) relative to $\mys$.
\end{enumerate}
When $E^0$ is finite, the procedure always terminates.
\end{proposition}
\begin{proof}
Assume first that \textsc{ProperlyInfinite} runs indefinitely. Applying  Lemma \ref{goldengoal}(iii) also to $(V_i,R_i)$ for $i>0$ we see that all sets $W_n\setminus E^0_2$ produced are disjoint, and since they may not be empty (or the procedure would terminate) we conclude that $E^0$ must be infinite. Using  Lemma \ref{goldengoal}(i) with a diagonal argument we can produce the sequence $(w_n)$ required in the definition of condition (AQ).
 
 If the procedure terminates with $\infty$, then we have that $(\emptyset,R')=$\textsc{FindFQ}$^{\circ n}(V,\emptyset)$, and since $p_{R'}$ is properly infinite by  Lemma \ref{infibasics}(i), so is $p_V$ by Lemma \ref{findFQi}. If the procedure terminates with $1$, the procedure has identified a $w_m\in W_m$ in either $E^0_0\setminus \mys$ or $E^0_1$. Again by Lemma \ref{goldengoal}(i), we establish condition (FQ) relative to $\mys$. 
\end{proof}

\begin{theorem}\label{Thm:notPI}
The conditions (i)--(iii) of Theorem \ref{Thm:notPI:pre} are equivalent to
\begin{enumerate}[(i)]\addtocounter{enumi}{3}
\item $V$ does not have property (FQ)
\end{enumerate}
\end{theorem}
\begin{proof}
The implication (iv)$\Longrightarrow$(iii) 
is Proposition~\ref{prop:stablyfinite}, and  
 by Proposition~\ref{Prop:flowchart}, 
we get (i)$\Longrightarrow$(vi). Indeed, if $p_V$ is not properly infinite, then \textsc{ProperlyInfinite} must terminate with 1 or run indefinitely, in which case it must have property (FQ) relative to $\mys$.
\end{proof}

\begin{example}\label{sumexxs}
Consider the graph 
\[\rule{0mm}{1cm}
\xymatrix{{\bullet}\ar@(u,r)[]\ar@(u,l)[]&{\bullet}\ar@(u,r)[]\ar@(u,l)[]\\{\circ}\ar[u]^(0.1){w_0}^(0.9){v_0}&{\circ}\ar[u]_(0.1){w_1}_(0.9){v_1}}
\]
with vertices $\{v_0,w_0,v_1,w_1\}$ as indicated. We see that $p_V$ fails to have property (FQ) precisely when 
\[
w_i\in V\Longrightarrow v_i\in V.
\]
Our theorem thus shows that except for the case ruled out just after Definition \ref{pidef}, all possibilities of proper infiniteness of a sum
relative to proper infiniteness of the summands occur in this
$C^*$-algebra. Indeed
\begin{center}
\begin{tabular}{|c|c|c||l|}\hline
$p$&$q$&$p+q$&Example\\\hline \hline
$\lnot$PI&$\lnot$PI&$\lnot$PI&$p=p_{\{w_0\}},q=p_{\{w_1\}}$\\\hline
$\lnot$PI&$\lnot$PI&PI&$p=p_{\{v_0,w_1\}},q=p_{\{v_1,w_0\}}$\\\hline
$\lnot$PI&PI&$\lnot$PI&$p=p_{\{w_0\}}, q=p_{\{v_1,w_1\}}$\\\hline
$\lnot$PI&PI&PI&$p=p_{\{w_0\}}, q=p_{\{v_0,v_1,w_1\}}$\\\hline
\end{tabular}
\end{center}
\end{example}

\section{Semiprojectivity of unital graph  $C^*$-algebras}

\subsection{Sufficiency}

In this section we provide a sufficient criterion for semiprojectivity of unital graph algebras, which  we in the ensuing section will see is also necessary. Precisely, we will prove

\begin{theorem}\label{suf}
Let $E=(E^0,E^1,s,r)$ be a graph such that $E^0$ is finite. 
If, for each $v\in E^0$, the set
\[
\Omega_v=\{w\in E^0\mid |vE^1w|=\infty\}
\]
does not have (FQ) relative to $\mys$, then the relative graph algebra $C^*(E,\mys)$ is semiprojective.
\end{theorem}

We base our work on the fundamental observation by Szyma\'nski

\begin{proposition}\label{fissp}
  When $E$ is a finite graph (i.e., $|E^0|,|E^1|<\infty$) the
  $C^*$-algebra $\toep{E}$ is semiprojective.
\end{proposition}

With this, we easily get the result (essentially proved in \cite{ws:scdg})  that any relative graph algebra $C^*(E,\mys)$ with $E$ finite is
semiprojective by appealing to the following result:

\begin{proposition}\label{qissp}
If $|E^0|<\infty$ with $\mys\subseteq E^0\reg$, let $I\gidealof C^*(E,\mys)$ be a gauge invariant
ideal. If $C^*(E,\mys)$ is semiprojective, then so is $C^*(E,\mys)/I$.
\end{proposition}
\begin{proof}
As seen in \cite{tbdpirws:crg} $I$ is generated by a set of vertex projections, which in this case must be 
finite, say, $p_1,\dots,p_k$.  With a map
$\phi:C^*(E,\mys)/I\to D/\overline{\cup J_k}$ given, we may partially lift
$\phi\circ\pi$ as indicated below
\[
\xymatrix{
{0}\ar[r]&{I}\ar[r]^-{\iota}&{C^*(E,\mys)}\ar[r]^-{\pi}\ar[d]_-{\psi}&{C^*(E,\mys)/I}\ar[r]\ar[d]_-{\phi}&0\\
&&{D/J_k}\ar[r]_{\kappa_k}&{D/\overline{\bigcup J_k}}}
\]
Since $\kappa_n\circ \psi\circ \iota (p_j)=0$ we may assume, by
increasing $n$ if necessary, that in fact $ \psi\circ \iota
(p_j)=0$. This implies that $\psi$ descends to $C^*(E,\mys)/I$, as desired.
\end{proof}

We frequently use, without mentioning, 
an easy fact that for a projection $q$ 
and two partial isometries $s_1, s_2$, 
the inequality $s_1s_1^*+s_2s_2^* \leq q$ 
is equivalent to the three equalities 
$qs_1=s_1$, $qs_2=s_2$ and $s_1^*s_2=0$. 

 The following result is a further refinement of a key result in Spielberg's refinement of Blackadar's proof that $\mathcal O_\infty$ is semiprojective.

\begin{proposition}\label{XX}
Let $B,D$ be \CA s, 
and $\pi\colon D \to B$ be a surjective \shoM . 
Let $p,q \in B$ be projections, 
and $u_0,u_1,u_2\in B$ be partial isometries satisfying
\[
u_0^*u_0=u_1^*u_1=u_2^*u_2=p,\qquad 
u_0u_0^*+u_1u_1^*+u_2u_2^* \leq q
\]
Let $P,Q \in D$ be projections with $\pi(P)=p$ and $\pi(Q)=q$. 
Let $U_1 \in D$ be a partial isometry 
with $\pi(U_1)=u_1$, $U_1^*U_1=P$ and $U_1U_1^*\leq Q$. 

Suppose that $P$ is properly infinite. 
Then there exist partial isometries $U_0,U_2 \in D$ 
satisfying $\pi(U_0)=u_0$, $\pi(U_2)=u_2$, and 
\[
U_0^*U_0=U_2^*U_2=P,\qquad 
U_0U_0^*+U_2U_2^* \leq Q.
\]
\end{proposition}

\begin{proof}
By elementary functional calculus, 
one can find $T_0,T_2 \in D$ such that 
$\pi(T_i)=u_i$ and 
\[
T_iP=T_i,\quad T_iT_i^* \leq Q - U_1U_1^*
\]
for $i=0,2$ and $T_0^*T_2=0$. 
Since $P$ is a properly infinite projection, 
there exist $V_0,V_2 \in D$ satisfying
\[
V_0^*V_0=V_2^*V_2=P,\quad 
V_0V_0^*+V_2V_2^*\leq P. 
\] 
Set 
\[
U_i = T_i+U_1V_i\sqrt{P-T_i^*T_i}
\]
for $i=0,2$. 
Since 
\[
T_0^*T_2=T_0^*U_1=T_2^*U_1=0
\]
and $V_0^*U_1^*U_1V_2=V_0^*PV_2=0$, 
we have $U_0^*U_2=0$. 
For $i=0,2$, we have $\pi(U_i)=u_i$ 
because $\pi(P-T_i^*T_i)=p-u_i^*u_i=0$. 
Since $T_i^*U_1=0$, 
we have
\[
U_i^*U_i
=T_i^*T_i+
\sqrt{P-T_i^*T_i}V_i^*U_1^*U_1V_i\sqrt{P-T_i^*T_i}\\
=P
\]
for $i=0,2$. 
Finally, it is easy to see $QU_i=U_i$ for $i=0,2$. 
\end{proof}

\begin{proposition}\label{Prop:Key1}
Let $E$ be a graph, $F$ a subgraph, and $S\subseteq E^0\reg$.
Suppose that 
a finite subset $V \subseteq F^0 \subseteq E^0$ does not satisfy (FQ)
relative to $\mys$ in $F$
and a vertex $w \in F^0$ satisfies
that $|wE^1v|=\infty$, $|wF^1v|\geq 1$ for all $v \in V$. 
Let $G$ be the subgraph of $E$ 
obtained by adding the infinitely many edges $\cup_{v \in V}wE^1v$ to $F$. 

Then $C^*(F,\mys_F) \subseteq C^*(G,\mys_{G})$ 
has weak relative lifting property. 
\end{proposition}
\begin{figure}
\begin{center}
\includegraphics[width=5cm]{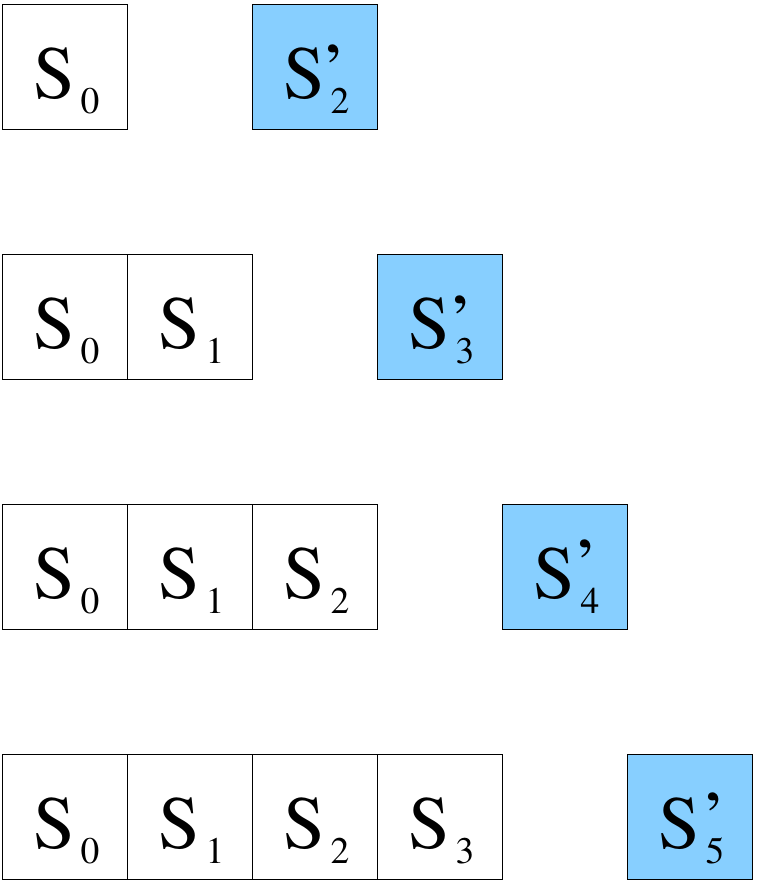}
\end{center}
\caption{Proof idea}
\end{figure}
\begin{proof}
Take \CA s $B,D$, a surjective \shom $\pi\colon D \to B$, 
a \shom $\varphi\colon C^*(G,\mys_G) \to B$ 
and a \shom $\psi\colon C^*(F, S_F) \to D$ 
such that $\pi \circ \psi = \varphi|_{C^*(F,\mys_F)}$. 
For each $v \in V$, 
enumerate $wE^1v=\{e_{v,n}\}_{n=0}^\infty$ 
such that $e_{v,1} \in F^1$. 
%%%%%%%%%%%%%%%%%%%%%%%
In order to construct a \shom $\overline{\varphi}\colon C^*(F',\mys_{F'}) \to D$ 
such that $\pi \circ \overline{\varphi} = \varphi$, 
we only need to find partial isometries\SEQ{Had a note about "adding just one edge"?}
$\{S_n\}_{n=0}^\infty$ in $D$ with orthogonal ranges 
satisfying 
\[
\pi(S_n)=\varphi(s_{e_n}),\quad 
S_n^*S_n=\psi(p_v),\quad 
S_nS_n^* \leq \psi(p_w)
\]
for all $n$. 
Set $S_1' = \psi(s_{e_1}) \in D$. 
Since, by Theorem \ref{Thm:notPI}, $p_V$ is properly infinite, we may use Proposition \ref{XX} with $P=\psi(p_V)$, $Q=\psi(p_w)$, $U_1=S_1'$, 
$u_0=\varphi(s_{e_{0}})$ and $u_2=\varphi(s_{e_{2}})$, 
to get $U_0,U_2 \in D$ as in its conclusion. 
Set $S_0 = U_0$ and $S_2' = U_2$. 
Then $S_0, S_2' \in D$ satisfies 
$\pi(S_0)=\varphi(s_{e_0})$, $\pi(S_2')=\varphi(s_{e_2})$, and 
\[
S_0^*S_0={S'_2}^*S'_2=\psi(p_v),\qquad 
S_0S_0^*+S_2{S'_2}^* \leq \psi(p_w).
\]
By induction on $n$, 
we are going to construct 
$S_0,S_1,\ldots, S_{n-1},S_{n+1}' \in D$ satisfying 
$\pi(S_i)=\varphi(s_{e_i})$ for $i=0,1,\ldots, n-1$, 
$\pi(S_{n+1}')=\varphi(s_{e_{n+1}})$, and 
\[
S_i^*S_i=(S'_{n+1})^*S'_{n+1}=\psi(p_v),\qquad 
\sum_{i=0}^{n-1}S_iS_i^*+S_{n+1}'(S'_{n+1})^* \leq \psi(p_w).
\]
For $n=1$, we are done. 
Suppose the claim is true for $n$.
To prove it for $n+1$, 
we need to construct $S_{n}$ and $S_{n+2}'$ with the desired properties. 
For this, 
apply Proposition \ref{XX} 
for $P=\psi(p_v)$, $Q=\psi(p_w)-\sum_{i=0}^{n-1}S_iS_i^*$, 
$U_1=S_{n+1}'$, 
$u_0=\varphi(s_{e_{n}})$ and $u_2=\varphi(s_{e_{n+2}})$ 
to get $U_0,U_2 \in D$, 
and set $S_n = U_0$ and $S_{n+2}' = U_2$. 
It is routine to check that these work. 
Thus we get the desired family $\{S_n\}_{n=0}^\infty$ in $D$. 
\end{proof}

\begin{proofof}{Theorem \ref{suf}}
Enumerate $E^0=\{v_1,\dots,v_n\}$ and, for
each $w\in \Omega_{v}$, all  edges in $wE^1v$ as
\[
\{e_{w,v}^k\}_{k=0}^\infty
\]
We define subgraphs $F_0\subseteq F_1\subseteq \dots \subseteq F_n=E$
all with vertex set $E^0$ and with edge sets defined by
\begin{eqnarray*}
F_0^1&=&E^1\setminus \{e_{w,v}^k\mid w\in \Omega_{v}, k\geq 2\}\\
F_j^1&=&F_{j-1}^1\cup  v_jE^1, j\in\{1,\dots, n\}
\end{eqnarray*}
By Proposition \ref{fissp}, we have that $C^*(F_0,\emptyset)$ is
semiprojective, and hence, by Proposition \ref{qissp}, so is
$C^*(F_0,\mys_{F_0})$. Thus we may prove the result by appealing to Lemma
\ref{weakone} and \ref{weaktwo} as soon as we have established that
each inclusion $C^*(F_i,\mys_{F_i})\subseteq C^*(F_{i+1},\mys_{F_{i+1}})$
has weak relative lifting property. This follows from 
Proposition
\ref{Prop:Key1}. 
\end{proofof}

\subsection{Necessity}

We now prove that the condition established to be sufficient for semiprojectivity in the previous section is in fact necessary for weak semiprojectivity.

\begin{theorem}\label{nec}
Let $E=(E^0,E^1,s,r)$ be a graph such that $E^0$ is finite. 
If, for some $v\in E^0$, the set
\[
\Omega_v=\{w\in E^0\mid |wE^1v|=\infty\}
\]
has (FQ) relative to $\mys$, then the relative graph algebra $C^*(E,\mys)$ is not weakly semiprojective with respect to the class of unital graph $C^*$-algebras.
\end{theorem}

We model our proof on the lack of weak semiprojectivity of $\bK^{\sim}$
and $(C(\T)\otimes\bK)^{\sim}$, which are graph $C^*$-algebras given
by
\[
\xymatrix{
{\circ}\ar@{=>}[r]&{\circ}&&{\bullet}\ar@{=>}[r]\ar@(ld,lu)[]&{\circ}}
\]
respectively.
 Indeed, the (FQ)
condition of some set $\Omega_v$ will allow us to find a quotient $C^*(F)$ of $C^*(E,\mys)$ with
properties similar to one of these examples, and then by Proposition
\ref{qissp}, we can conclude that $C^*(E,\mys)$ is not weakly semiprojective.
In fact, as we shall see later, any non-semiprojective unital graph algebra contains one of these algebras as a subquotient, up to Morita equivalence. %We do not know how to see this directly.

\begin{proof}
We fix $v \in E^0$ such that $\Omega_v$ satisfies (FQ) relative to
$\mys$. Since $E^0$ is finite, this means that there is some $w\in E^0$
such that either $\Omega_v$ has property (MQ) at $w\not\in S$ or $\Omega_v$ has property (TQ) at $w\in E^0$. In either case, by setting
\[
F=s(E^\dagger w)
\] 
and considering the \nap $(E\setminus F,E^0\reg)$ we can pass to a quotient $C^*(F)$ with $v\in F$ and $V=\Omega_v\cap F\subseteq F$ such that $p_V\in C^*(F)$ is nonzero and stably finite and such that for each $x\in V$, $|vF^1x|=\infty$.
We now define subgraphs $G_n$ of $F$, by 
enumerating the infinitely many vertices from $x\in V$ to $v$ by $\{e_x^k\}_{k\in \N}$ and letting $G_n$ be the graph with vertex set $E^0$ and edge set
\[
G_n^1=E^1\setminus \{e_w^k\mid w\in V, k>n\}
\]

There is a $*$-homomorphism
\[
\Xi:C^*(F)\to\frac{\prod C^*(G_n,\mys_{G_n})}{\sum C^*(G_n,\mys_{G_n})}
\]
defined by 
\[
\Xi(p_v)=[(p_v,p_v,\dots)]\qquad
\Xi(s_{e})=\begin{cases}[(\overbrace{0,\dots,0}^k,s_{e_w^k},s_{e_w^k},\dots)]&e=e_w^k\\
[(s_e,s_e,\dots)]&\text{other }e
\end{cases}
\]
If $C^*(E)$ is weakly semiprojective with respect to the class of unital graph $C^*$-algebras, there is a map $\Phi=(\phi_n):C^*(E)\to 
\prod C^*(G_n,\mys_{G_n})$ lifting $\Xi\circ \kappa$. Note that for any projection $p\in C^*(E)$ with $\kappa(p)=0$ we have $\Xi(p)\in \sum_{n=1}^\infty 
C^*(G_n,\mys_{G_n})$ and hence $\Xi(p)=
\sum_{n=1}^N
C^*(G_n,\mys_{G_n})$ for some $N$ depending on $p$. But since $\ker \kappa$ is generated by a finite number of projections, we may in fact assume, by replacing a number of coordinate maps $\phi_n$  by zero, that $\Phi(\ker\kappa)=0$. Let $\Psi=(\psi_n)$ denote the induced map, so that we now have
\[
\xymatrix{
{C^*(E)}\ar[r]^-{\kappa}\ar[d]_{\Phi}&{C^*(F)}\ar[d]_{\Xi}\ar[ld]_{\Psi}\\
{\prod_{n=1}^\infty 
C^*(G_n,\mys_{G_n})}\ar[r]_-{\pi}&{\dfrac{\prod_{n=1}^\infty 
C^*(G_n,\mys_{G_n})}{\sum_{n=1}^\infty 
C^*(G_n,\mys_{G_n})}}}
\]
We now consider the two projections $p_V$ and
\[
q=p_w-\sum_{e\in w E^1\setminus E^1V}{s_es_e^*}
\]
which may be considered as elements both of $C^*(F)$ and $C^*(G_n,\mys_{G_n})$ for any $n$ since none of the edges in 
$\{e_x^k\}_{k\in \N}$ are involved.

Since $\Psi$ is a lift of $\Xi$ there is an $N$ such that
\[
\left\|p_V-\psi_N(p_V) \right\|<1\qquad
\left\|q-\psi_N(q) \right\|<1
\]
Note that $m[p_V]\leq [q]$ in $V(C^*(F))$ for any $m$, so that 
\[
(N+1)[\psi_N(p_V)]=(N+1)[p_V]\leq [q]=[\psi_N(q)]
\]
in $V(C^*(G_n,\mys_{G_n}))$. The relation holds true also in $V(C^*(G_n))$ after passing to quotients. But here, we further have
$
[q]\leq N[p_V]
$
so we get 
\[
(N+1)[p_V]\leq N[p_V]
\]
which contradicts the stable finiteness of $p_V$.
\end{proof}

\subsection{Semiprojectivity criteria}

Combining the results above, we now present the first installment of our main results on semiprojectivity of unital graph $C^*$-algebras. In the last section of the paper, we provide even more equivalent conditions to the striking effect that semiprojectivity in a certain sense can be checked ``locally''. But for now, we prove

\begin{theorem}\label{Thm:finite}
Let $E=(E^0,E^1,s,r)$ be a graph such that $E^0$ is finite. 
For each $v\in E^0$, set
\[
\Omega_v=\{w\in E^0\mid |wE^1v|=\infty\}.
\]
The following are equivalent
\begin{enumerate}[(i)]
\item $C^*(E,\mys)$ is semiprojective 
\item $C^*(E,\mys)$ is weakly semiprojective 
\item $C^*(E,\mys)$ is weakly semiprojective w.r.t. the class of unital graph $C^*$-algebras
\item  For each $v$, $\Omega_v$
does not have (FQ) relative to $\mys$.
\item  For each $v$, $p_{\Omega_v}$ is properly infinite.
\end{enumerate}
\end{theorem}
\begin{proof}
Combine Propositions \ref{nec}, \ref{suf} and \ref{Thm:notPI}.
\end{proof}

Note that since a unital and simple graph $C^*$-algebra is either $M_n(\C)$ or is purely infinite, we have reproved Szyma\'nski's result (\cite{ws:scdg}) that any such $C^*$-algebra is semiprojective. We also easily get the generalization

\begin{corollary}
A purely infinite and unital graph $C^*$-algebra $C^*(F)$ is semiprojective.
\end{corollary}
\begin{proof}
By Theorem \ref{Thm:finite}(v), we need to check that certain projections in $C^*(F)$  are properly infinite. But in this case, all projections have this property, cf. \cite[4.16]{ekmr:npic}.
\end{proof}

The reader is asked to note that if $C^*(F)$ given above has only finitely many ideals, it has real rank zero. In fact, we do not know of a semiprojective $C^*$-algebra with finitely many ideals which fails to have real rank zero.

\begin{example}\label{recexxii}
Consider the graph $C^*$-algebras given by the graphs $\myi$--$\myvi$ in Example \ref{recexx}. We note that the set $\{w_1\}$ has (FQ) in all of these graphs, and hence that in all cases except $\myv$, the $C^*$-algebras are not semiprojective. In the case $\myv$ we need instead to check whether $\{v_0,w_1\}$ has (FQ), and indeed this is not the case. 
Consequently, $C^*(\myv)$ is in fact semiprojective.
\end{example}

\begin{remark}
Note that as a consequence of Proposition \ref{fissp}, if the unital Toeplitz algebra $\toep{E}$ is semiprojective, so is any other relative graph algebra $C^*(E,\mys)$. Conversely,  if the graph algebra $C^*(E)$ is not semiprojective, neither is any other relative graph algebra on $E$. Intermediate cases occur; consider for instance the graph
\[
\xymatrix{
\circ\ar@(lu,ld)[]\ar@(u,l)[]\ar[r]_(0.2){v_1}&\circ\ar@{=>}[r]_(0.2){v_2}_(1.2){v_3}&\circ}
\]
where we have four options for choosing which instances of \CK{3} to impose: $\mys_1=\emptyset$, $\mys_2=\{v_1\}$, $\mys_3=\{v_2\}$ and $\mys_4=\{v_1,v_2\}$. We get directly from the (FQ) criterion at $\Omega_{v_3}=\{v_0\}$ that the algebra is semiprojective precisely when \CK{3} is imposed at $v_2$, so $C^*(E,\mys_1)=\toep{E}$ and $C^*(E,\mys_2)$ fail to be semiprojective whereas 
$C^*(E,\mys_3)$ and $C^*(E,\mys_4)$ will be.
\end{remark}

The commutative $C^*$-algebra $C_0(\N)$ is the non-unital graph algebra  given by a countable set of vertices with no edges. It is easy to see that it satisfies (ii)--(v) above, but not (i). For general non-unital graph $C^*$-algebras we may get infinite sets $\Omega_v$ so that in fact (iv) is not defined and (v) involves a projection in the multiplier algebras, no it is not at this stage clear how to generalize properties (iv) and (v) to the non-unital setting, and 
we will see below that (i) is not equivalent to (ii) for non-unital graph $C^*$-algebras, even when these are simple.

We may, however, resolve the semiprojectivity issue for a general Toeplitz algebra by combining our result with the  following well-known facts

\begin{proposition}\label{infeasy}
Let $E$ be a graph and $\mys\subseteq E^0\reg$.
\begin{enumerate}[(i)]
\item If $K_*(C^*(E,\mys))$ is not finitely generated, then $C^*(E,\mys)$ is not semiprojective
\item $K_0(\toep{E})=\Z^{E_0}$
\end{enumerate}
In particular, $\toep{E}$ can only be semiprojective when $|E^0|<\infty$.
\end{proposition}
\begin{proof}
Using Lemma \ref{subgraph} we can write $C^*(E,\mys)$ as an increasing union of relative graph algebras $C^*(F_i,\mys_{F_i})$ associated to finite graphs. These have finitely generated $K$-theory, and had $C^*(E,\mys)$ been semiprojective, we would have that $K_*(C^*(E,\mys))\simeq K_*(C^*(F_i,\mys_{F_i}))$ for some $i$, which is a contradiction. The second claim is noted, e.g., in \cite{tmcsemt:imkga}.
\end{proof}

Spielberg \cite{js:scpic} has proved that any purely infinite and simple graph $C^*$-algebra with finitely generated $K$-theory is in fact semiprojective, so again we see that semiprojectivity of relative graph algebras varies dramatically with the choice of $\mys$.

A graph $C^*$-algebra $\overline{E}$ is \emph{amplified} when whenever there is an edge from $v$ to $w$ in $\overline{E}$, there are infinitely many. In other words, all entries of the adjacency matrix lie in $\{0,\infty\}$. This class of $C^*$-algebras is studied in \cite{seerapws:agc}, associating to $\overline{E}$, in an obviously bijective fashion, the simple graph $E$ with each collection of edges replaced by one representative. Our criterion here simplifies to the following result, which was independently observed by S\o rensen and Spielberg:

\begin{corollary}
The following  are equivalent for any graph $E$
\begin{enumerate}[(i)]
\item $C^*(\overline{E})$ is semiprojective
\item $E^0$ is finite, and for each $v$ and $w$ in $E$, if there is a path from $v$ to $w$ of length $\ell$, there is also a path of length $>\ell$.
\end{enumerate}
\end{corollary}
\begin{proof}
Since every vertex is singular, we have $C^*(\overline{E})=\toep{\overline{E}}$, so by Proposition \ref{infeasy} can not be semiprojective when $|E^0|=\infty$. Let us therefore assume that $|E^0|<\infty$ so that $C^*(\overline{E})$ is unital. Suppose now that there is a path $e_1\cdots e_\ell$ in $E$ with $s(e_1)=v$ and $r(e_\ell)=w$, but no longer path. This property passes to subpaths, so we may assume that $\ell=1$. Note also that $v,w\in E^0_0$. 
Turning to $\overline{E}$, we see that $v\in \Omega_w$, and since there can be no indirect path from $v$ to $w$, we get that $v\in{\textsc{GetTop}}(\Omega_v)$. Since $v\not \in \mys=\emptyset$, we have ${\textsc{FindFQ}}(\Omega_v,\emptyset)=(\star,\star)$.

In the other   direction, assume that $\Omega_v$ has property (FQ) for some $v$. Since the other options are ruled out, $\Omega_v$ must have (MQ) with respect to some $w$. Hence there are only finitely many paths from $w$ to $v$, one of which is longest.
\end{proof}

Now consider the graph $E$ given as
\[
\xymatrix{
\bullet\ar@/^/[r]&\bullet\ar@/^/[r]\ar@/^/[l]&\bullet\ar@/^/[r]\ar@/^/[l]&\bullet\ar@/^/[r]\ar@/^/[l]&\bullet\ar@/^/[r]\ar@/^/[l]&\cdots\ar@/^/[l]}
\]
We have seen in Proposition \ref{infeasy}(i) that $C^*(\overline{E})$ is not semiprojective because it has $K_0(C^*(\overline{E}))=\Z^\infty$. But since it is simple and purely infinite,
we get by \cite{hl:wspisc} or \cite{js:wspic} that it is weakly semiprojective. Consequently, property (ii)  of Theorem \ref{Thm:finite} is satisfied by $C^*(\overline{E})$, but not (i).

\begin{example}\label{alltwobytwo}
Consider all $2\times 2$-matrices with entries in $\{0,1,2,\infty\}$. Among these $4^4=256$  matrices, only the following 12 are not leading to semiprojective graph $C^*$-algebras
\begin{gather*}
\begin{bmatrix}0& 0\\ \infty& 0\end{bmatrix},
\begin{bmatrix}1& 0\\ \infty& 0\end{bmatrix},
\begin{bmatrix}2& 0\\ \infty& 0\end{bmatrix},
\begin{bmatrix}0& \infty\\ 0& 0\end{bmatrix},
\begin{bmatrix}1& \infty\\ 0& 0\end{bmatrix},
\begin{bmatrix}0& 0\\ \infty& 1\end{bmatrix}\\
\begin{bmatrix}1& 0\\ \infty& 1\end{bmatrix},
\begin{bmatrix}2& 0\\ \infty& 1\end{bmatrix},
\begin{bmatrix}0& \infty\\ 0& 1\end{bmatrix},
\begin{bmatrix}1& \infty\\ 0& 1\end{bmatrix},
\begin{bmatrix}0& \infty\\ 0& 2\end{bmatrix},
\begin{bmatrix}1& \infty\\ 0& 2\end{bmatrix}
\end{gather*}
Note that even though $\left[\begin{smallmatrix}2& 0\\ \infty& 0\end{smallmatrix}\right]$ is not semiprojective, the transpose $\left[\begin{smallmatrix}2& \infty\\0& 0\end{smallmatrix}\right]$ is. In general, there is no relation between the semiprojectivity of the $C^*$-algebras associated to a graph and its reversion.

Note also that except for the graph given by two matrices $\left[\begin{smallmatrix}0& \infty\\ 0&
  0\end{smallmatrix}\right],\left[\begin{smallmatrix}0&0\\ \infty&
  0\end{smallmatrix}\right]$ above, the non-semiprojective examples have
two or more ideals. In fact we shall see in Corollary \ref{oneideal} below that when there is precisely one non-trivial ideal in $C^*(E)$, the only non-semiprojective examples look very much like these two examples.
\end{example}

\section{Corners, subquotients and extensions}

Our general knowledge of closure properties for the class of semiprojective $C^*$-algebras is rather incomplete in spite of decades of interest. In the following section, we investigate the closure properties restricted to graph algebras and show that some of the conjectured closure properties fail even here. 

\subsection{Subquotients}

It is well known that ideals of semiprojective $C^*$-algebras are not semiprojective,  and the prime example is given by a gauge invariant ideal of a unital graph algebra: the standard Toeplitz algebra is semiprojective although it contains the ideal $\bK$ which is not. 
There is also an ample supply of non-semiprojective quotients of a semiprojective graph $C^*$-algebra such as $C(\T)$, cf.\ \cite{apwsht:cscc}. Of course 
the kernel in that case is never gauge invariant, and we have in fact seen in Lemma \ref{qissp} that when $I\gidealof C^*(E,\mys)$ with $C^*(E,\mys)$ semiprojective and unital, then also $C^*(E,\mys)/I$ is semiprojective.

In general, we do not know how to decide semiprojectivity of non-unital graph algebras, but our result can be extended to those non-unital graph $C^*$-algebras that happen to be gauge invariant  ideals of unital ones. Indeed, following \cite{kdjhhws:srgatigal} as amended in \cite{ermt:iga}, we will for any ideal given by an \nap $(H,\myr)$ consider the map
\[
\eta(v)=\begin{cases} E^*(\myr\backslash E^0) E^1 v&v\in H\\
E^*v&v\not\in H\end{cases}
\]
associating to each vertex either the set of paths starting in $v$ and leaving $\myr$ after one step in the case of an $v$ in $H$ or the the set of all paths starting at $v$ in the remaining relevant cases.  With this definition, \cite{ermt:iga} extends to the relative case as follows
\begin{theorem}\label{idealasgraph}
Let $E=(E^0,E^1,s,r)$ be a graph such that $E^0$ is finite and
consider $I\gidealof C^*(E,\mys)$ represented by the \nap $(H,\myr)$. 
Then $I$ is isomorphic to the graph $C^*$-algebra $C^*(F,\mys_F)$ with
\begin{eqnarray*}
F^0&=&H\cup (\myr\backslash\mys)\cup\bigcup_{v\in H\cup (\myr\backslash\mys)}\eta(v)\\
F^1&=&HE^1\cup (\myr\backslash \mys)E^1H\cup \bigcup_{v\in H\cup (\myr\backslash\mys)}\eta(v)\\
\mys_F&=&(H\cap \mys)\cup  \bigcup_{v\in H\cup (\myr\backslash\mys)}\eta(v)
\end{eqnarray*}
and $I^\sim$ is isomorphic to the graph $C^*$-algebra $C^*(G,\mys_G)$ with
\begin{eqnarray*}
G^0&=&H\cup (\myr\backslash\mys)\cup\{\infty\}\cup \bigcup_{v\in H\cup (\myr\backslash\mys), |\eta(v)|<\infty}\eta(v)\\
G^1&=&HE^1\cup (\myr\backslash \mys)E^1H\cup \bigcup_{v\in H\cup (\myr\backslash\mys)}\eta(v)\\
\mys_G&=&(H\cap \mys)\cup  \bigcup_{v\in H\cup (\myr\backslash\mys),|\eta(v)|<\infty}\eta(v)
\end{eqnarray*}
such that for $\xi\in \eta(v)$  have
\[
s_F(\xi)=s_G(\xi)=v\qquad r_F(\xi)=r_G(\xi)=\xi
\]
when $\eta(v)$ is finite, and
\[
s_F(\xi)=s_G(\xi)=v\qquad r_F(\xi)=\xi\qquad r_G(\xi)=\infty
\]
when $\eta(v)$ is infinite.
\end{theorem}
\begin{proof}
Passing from the relative setting to $C^*(\makestandard{E}{\mys})$, we obtain the first statement from \cite{ermt:iga}. For the second, we note that $C^*(G,\mys_G)$ is unital and has an \nap $(E^0\backslash \{\infty\},E^0\backslash \{\infty\})$. For the corresponding ideal $J$ we clearly have $C^*(G,\mys_G)/J=\C$, and $J$ is precisely given by the graph that we just saw represent $I$. Hence $C^*(G,\mys_G)=I^\sim$.
\end{proof}

In the result above $I^\sim$ is to be understood as forced unitization in the sense that $I^\sim=I\oplus \C$ when $I$ is already unital.

\begin{example}\label{recexxiii}
Recall that in graphs $\myii$--$\myvii$ from Example \ref{recexx} we could choose an \nap of the form either $(H,H)$ or $(H,H\cup\{v_0\})$. In all cases but the last there was only one possible choice of an \nap $(H,\myr)$, but in case $\myvii$ we could take both. Let us refer to the situation with $\myr=H$ as $\myviia$ and the other as $\myviib$. In all the cases $\myii,\myiii,\myvi$ and $\myviia$ we get by Theorem \ref{idealasgraph} that the corresponding ideals and their unitizations are given by the graphs 
\[
\xymatrix@C-=1mm{
&&&&&&&&&&\\
&&&\bullet\ar@(lu,ld)[]\ar@(l,d)[]\ar[u]_(0.2){w_0}\ar[ul]\ar[ull]\ar[ulll]^(0.6){\dots}&&&\circ\ar[lll]\ar[u]^(0.2){w_1}\ar[ur]\ar[urr]\ar[urr]\ar[urrr]_(0.6){\dots}&&&
}
\qquad
\xymatrix@C-=4mm{
&{\circ}&\\%\ar@(ul,ur)[]\ar@(l,u)&\\
\bullet\ar@(lu,ld)[]\ar@(l,d)[]\ar@{=>}[ur]_(0.3){w_0}^(1.1){\infty}&&\circ\ar[ll]\ar@{=>}[ul]^(0.3){w_1}
}
\]
In case $\myv$ we get that the ideals are given by 
\[
\raisebox{-1cm}[1.1cm]{}
\xymatrix@C-=1mm{
{\bullet}&&&&{}&&&\\
\bullet\ar@(lu,ld)[]\ar@(l,d)[]\ar[u]_(0.2){w_0}_(0.9){v_0}&&&&\circ\ar[llll]\ar[u]^(0.2){w_1}\ar[ur]\ar[urr]\ar[urrr]_(0.6){\dots}
}
\qquad
\xymatrix{
{\bullet}&{\circ}\\
\bullet\ar@(lu,ld)[]\ar@(l,d)[]\ar[u]_(0.2){w_0}_(0.9){v_0}&\circ\ar[l]\ar@{=>}[u]^(0.2){w_1}^(0.9){v_1}
}\bigskip
\]
and in case $\myviib$ by
\[
\xymatrix@C-=1mm{&&&&&&&&\\
&&{\circ}\ar[u]\ar[ur]\ar[urr]\ar[urrr]_(0.6){\dots}
&&&&&\\
\bullet\ar@(lu,ld)[]\ar@(l,d)[]\ar@{=>}[urr]_(0.3){w_0}^(1.1){v_0}&&&&\circ\ar[llll]\ar@{=>}[ull]^(0.3){w_1}&&&
}
\qquad
\xymatrix@C-=4mm{&\circ&\\
&{\circ}\ar@{=>}[u]_(0.9){\infty}
&\\%\ar@(ul,ur)[]\ar@(l,u)&\\
\bullet\ar@(lu,ld)[]\ar@(l,d)[]\ar@{=>}[ur]_(0.3){w_0}_(1.1){v_0}&&\circ\ar[ll]\ar@{=>}[ul]^(0.3){w_1}
}\bigskip
\]

\end{example}

In the previous example we saw how easy it was to compute a graph representation of  $I^\sim$, and since this is unital by construction we may apply our standard criterion to decide semiprojectivity of $I^\sim$. This also decides semiprojectivity of $I$ since Blackadar in \cite{bb:stc} proved that $I$ is semiprojective precisely when $I^\sim$ is. Formally:

\begin{theorem}\label{idealcor}
Let $E=(E^0,E^1,s,r)$ be a graph such that $E^0$ is finite and
consider $I\gidealof C^*(E,\mys)$ represented by the \nap $(H,\myr)$. 
Then $I$ is semiprojective precisely when none of the sets in
\[
\{\Omega_v\mid v\in H\cup(\myr\backslash\mys)\}\cup \{\Omega_\infty\}
\]
has (FQ) relative to $\mys$, where
\[
\Omega_\infty=\{v\in \myr\mid |\eta(v)|=\infty\}
\]
\end{theorem}
\begin{proof}
By \cite{bb:stc} we may focus our attention on $C^*(G,\mys_G)$ which is again covered by our main theorem.
We see that the infinite receivers of $G$ are precisely $\infty$ and those of $H\cup \myr\setminus\mys$, and that the set of relevance in  Theorem \ref{Thm:finite} in the former case is precisely $\Omega_\infty$.
\end{proof}

\begin{example}\label{recexxiv}
Collecting the information in Examples \ref{recexxii} and \ref{recexxiii} we get
\begin{center}
\begin{tabular}{|c||c|c|c|c|c|c|}\hline
&$\myii$&$\myiii$&$\myv$&$\myvi$&$\myviia$&$\myviib$\\\hline\hline
$I$& SP&SP&$\lnot$SP&SP&SP&$\lnot$SP\\\hline
$C^*(E_\blacksquare)$& $\lnot$SP&$\lnot$SP&$\lnot$SP&SP&$\lnot$SP&$\lnot$SP\\\hline
$C^*(E_\blacksquare)/I$& $\C^2$&$\mathcal O_2\oplus\C$&$\C$&$\mathcal O_\infty$ &$\mathcal E_2$&$\mathcal O_2$\\\hline
\end{tabular}
\end{center}
\end{example}

\subsection{Extensions}

Blackadar conjectured (\cite[4.5]{bb:ssc}) that when an extension
\[
\xymatrix{
{0}\ar[r]&{I}\ar[r]&{A}\ar[r]&{\C}\ar[r]&0}
\]
splits, the implication
\[
I\text{ semiprojective }\Longrightarrow A\text{ semiprojective}
\] 
would hold. More generally, Loring (\cite{tal:lsppc}) asked whether, when $\dim F<\infty$, and
\[
\xymatrix{
{0}\ar[r]&{J}\ar[r]&{B}\ar[r]&{F}\ar[r]&0}
\]
we have that 
\[
J\text{ semiprojective }\Longleftrightarrow B\text{ semiprojective}
\] 
Enders (\cite{de:sfcsc}) recently proved that the backward implication
holds. 

We have already, however, in Example \ref{recexxiv} seen an example $\myii$ where the answer to Loring's question was negative, and with a little more work we will see that also Blackadar's conjecture fails. In fact, we saw in the table of Example \ref{recexxiv} that such a behavior also occurs when the quotient is $\mathcal O_2\oplus \C$ or the Cuntz-Toeplitz algebra $\mathcal E_2$.
We now give a complete description of when we may, and may not, infer from semiprojectivity of a gauge invariant ideal to semiprojectivity of the algebra itself, in such a unital extension. 

\begin{definition}\label{stardef}
Let $C^*(F,\mys)$ be a relative graph $C^*$-algebra. We say that $v_0,v_1\in F^0$ satisfy property $(*)$ when
\begin{enumerate}[(1)]
\item $v_1\not\in \mys$;
\item $v_0,v_1\not\in s(\Omega_{v_1}F^*)$; and
\item \begin{enumerate}[(a)]
\item $F^*v_0$ is infinite, or
\item $v_0\not=v_1$ and $v_0\not\in \mys$ 
\end{enumerate}
\end{enumerate}
\end{definition}

\begin{theorem}\label{bbwaswrong}
Let $C^*(F,\mys)$ be a unital relative graph algebra. The following are equivalent:
\begin{enumerate}[(i)]
\item Whenever a gauge invariant extension 
\[
\xymatrix{
{0}\ar[r]&{I}\ar[r]&{C^*(E,\mys_E)}\ar[r]&{C^*(F,\mys)}\ar[r]&0},
\]
given by an \nap $(H,\myr)$ with $E\backslash H=F$, $\myr\backslash H=\mys$ has $I$ is semiprojective, then so is $C^*(E,\mys_E)$.
\item $C^*(F,\mys)$ is semiprojective, and no pair of vertices $v_0,v_1\in F^0$ satisfy property $(*)$.
\end{enumerate}
\end{theorem}
\begin{proof}
We prove (i)$\Longrightarrow$ (ii) by contraposition. If $C^*(F,\mys)$ is not semiprojective, then
\[
\xymatrix{
{0}\ar[r]&{0}\ar[r]&{C^*(F,\mys)}\ar[r]&{C^*(F,\mys)}\ar[r]&0}
\]
is a counterexample establishing the negation of (ii). When $C^*(F,\mys)$ is semiprojective, we may choose a pair of vertices $v_0$ and $v_1$ satisfying $(*)$ and produce a counterexample as follows, taking the lead from Examples \ref{recexxiii} and \ref{recexxiv}. Indeed, we will produce $E$ by adding  the graph $H$ given as there by
\[
\xymatrix{
\bullet\ar@(lu,ld)[]\ar@(l,d)[]&\circ\ar[l]_(0.8){w_0}_(0.2){w_1}}\bigskip
\]
In case (3b) of (*) we add infinitely many edges from $w_0$ to $v_0$ and infinitely many edges from $w_1$ to $v_1$ (as in our example $\myii$) and in case (3a) we add a single edge from $w_0$ to $v_0$ and infinitely many edges from $w_1$ to $v_1$ (as in our examples $\myii,\myiii$ and $\myviia$). With $E$ thus defined, we set $\mys_E=\mys\cup\{w_0\}$ as indicated.

Employing our assumption (1), we see that in any case $(H,H\cup \mys)$ becomes an \nap since $v_i$ will only receive infinitely in the case that $v_i\not\in \mys$,  and as in Example  \ref{recexxiii} we get that the ideal thus defined is semiprojective. Hence we now only need to prove that $C^*(F,\mys_F)$ is not semiprojective. To see this, note that $w_1\in \Omega_{v_0}$ and that by our assumption (2), $w_0\not\in s(\Omega_{v_0}F^*)$. Thus $w_1$ remains in \textsc{FindFQ}$^{\circ n}(\Omega_{v_0},\emptyset)$ for all $n$, and we see that \textsc{ProperlyInfinite}$(\Omega_{v_0})$ can not return $(\star,\star)$. Consequently, $p_{\Omega_{v_0}}$ is not properly infinite, and Theorem \ref{Thm:finite} applies.

In the other direction, suppose no pair of vertices has property (*) in $C^*(F,\mys)$, and that $C^*(F,\mys)$ as well as $I$ are semiprojective. We must check that $\Omega_v$ fails to have (FQ) for every $v\in E$. This follows directly from the semiprojectivity of $I$ for any $v\in H$, so we now only consider the case of $v\in F$.  Consider first the case $v\in \mys$. Assuming that $\Omega_v\not=\emptyset$, we would of course have that $v\in E^0\sing$ and hence that  $v\in\myr\backslash H$, and by our criterion deciding in Theorem \ref{idealcor} that $I$ is semiprojective, further that $\Omega_v$ does not have (FQ). In the case when further $F^*v$ is finite, we may in fact conclude that $\Omega_v=\emptyset$, for if not we would have that $v\in\Omega_\infty$ with $v\in H^0_0\backslash \mys$ contained in every top of the set, just as in the case $\myviib$.

Thus we may focus on the case $v_1\not \in \mys$. If $v_1F^* w_1\not=\emptyset$, $w_i\in \Omega_{v_1}$, we get that $\{w_1\}$ is a top for $\Omega_{v_1}$, and that $w_1\in E^0_2$. Hence also here we get that $\Omega_{v_1}$ does not have (FQ), and we may now assume that $v_1\not\in s(\Omega_{v_1}F^*)$. Since we know that $(v_1,v_1)$ does not have (*) we conclude that $F^*v_1$ is finite, and since $(v_0,v_1)$ does not have (*) for any other $v_0$ we conclude that one of
\begin{enumerate}[(A)]
\item $v_0\in s( \Omega_{v_1}F^*)$
\item $v_0\in \mys$ and $F^*v_0$ is finite
 \end{enumerate}
holds. We now note that $\Omega_{v_1}\cap H\subseteq \Omega_\infty$, and that for any $w\in \Omega_\infty\backslash \Omega_{v_1}$, we have $w\in \Omega_{v_1}F^*$. Indeed, as we have seen above, it is not possible to find an infinite number of paths starting at $w$, leaving $H$ after one step, and visiting only vertices in case (B), since they have finite futures and can only be reached by an edge with finite multiplicity. It is also not possible to find an infinite number of paths via $v_1$, so at least one vertex from our case (A) must be visited, and then of course there is a path to $\Omega_v$ as well.

Suppose now for contradiction that $\Omega_{v_1}$ has (MQ), so that $x\not\in \mys_E$ and that $\Omega_vE^*x$ is a finite, nonempty set. If we had $(\Omega_v\cap H)E^*x=\emptyset$ then we would have $(\Omega_v\cap F)E^*x\not=\emptyset$ and consequently  $x\in \Omega_\infty$. Thus, in any case, $\Omega_\infty E^*x\not=\emptyset$. But since we know that $\Omega_\infty$ does not have (MQ), we must have that $\Omega_\infty E^*x$ as well as $(\Omega_\infty\backslash \Omega_{v_1}) E^*x$ infinite. We have, however, seen that $\Omega_{v_1}E^*(\Omega_\infty\backslash \Omega_{v_1})\not=\emptyset$, so this provides the desired contradiction. Proving that (TQ) does not occur follows similarly.
\end{proof}

It is easy to check that when $C^*(F,\mys)$ is gauge simple or regular (i.e. $F^0=F^0\reg=\mys)$ -- in particular, if it is a Cuntz-Krieger algebra -- then it has the properties of (ii) above, and that  when
\begin{itemize}
\item $\mys\not=F^0\reg$; or
\item $|F^0|>1$, and $F$ has a source which is also a sink
\end{itemize}
it has not. In particular, as we noted above, there are counterexamples to Loring's question even for $\C^2$, but not for $\C$. This would seem to corroborate Blackadar's conjecture, but in fact -- as was well known to him -- the fact that a unital extension by $\C$ preserves semiprojectivity follows directly from the observation that we have already used from \cite{bb:ssc}. It is interesting to note that for unital extensions in our setting, if the quotient is $M_2(\C)$, $\mathcal{O}_2$, $\mathcal{O}_2\oplus \mathcal{O}_2$, or $\mathcal{O}_\infty$, we may infer from the semiprojectivity of the ideal to semiprojectivity in the middle.  In particular, we get by combining our results with those of \cite{de:sfcsc}:

\begin{corollary}
When 
\[
\xymatrix{
{0}\ar[r]&{I}\ar[r]&{C^*(E,\mys)}\ar[r]&{M_n(\C)}\ar[r]&0}
\]
is exact and unital with $I$ a gauge invariant ideal, we have that
\[
I\text{ semiprojective }\Longleftrightarrow C^*(E,\mys)\text{ semiprojective}
\] 
\end{corollary}

It is conceivable that the corollary holds without the condition that $I$ is a gauge invariant ideal.
If we allow non-unital extensions the situation deteriorates completely:

\begin{corollary}\label{bbwaswrongagain}
Suppose $C^*(F,\mys)$ is a unital graph $C^*$-algebra with $|F^0|>0$. There exists a graph $E$ with $|E^0|=\infty$ and a gauge invariant ideal $I\gidealof C^*(E,\mys)$ such that 
$C^*(E,\mys)/I= C^*(F,\myr)$,
\[
\xymatrix{
{0}\ar[r]&{I}\ar[r]&{C^*(E,\mys)}\ar[r]&{C^*(F,\myr)}\ar[r]&0}
\]
splits, $I$ is {semiprojective}, and $C^*(E,\mys)$ is not {semiprojective}. 
\end{corollary}
\begin{proof}
We may choose $v_0\in F^0$ so that either $v_0\not\in \mys$ or $F^*v_0$ is infinite. We now construct the graph $E'$ by adding the graph 
\[
\xymatrix{
\bullet\ar@(lu,ld)[]\ar@(l,d)[]&\circ\ar[l]_(0.8){w_0}_(0.2){w_1}\ar@{=>}[r]^(1.2){v_1}&\circ}\bigskip
\]
to $F$, placing one edge from $w_0$ to $v_0$ when $F^*v_0$ is infinite, or infinitely many edges from $w_0$ to $v_0$ when it is not. Considering the ideals $I,J$ given by the \naps
\[
(\{w_0,w_1\},\mys)\qquad
(F\cup\{w_0,w_1\},\mys)
\]
we get that $J/I\simeq C^*(F,\mys)$ with $I$ being semiprojective and $J$ failing to be so. We then choose $(E,E_\mys)$ realizing $J$ according to Theorem \ref{idealasgraph}. Noting that there are elements in $C^*(E,\mys)$ satisfying the Cuntz-Krieger relations for $C^*(F,\myr)$ we get the stipulated splitting map. 
\end{proof}

Our method in fact exclusively allows the construction of split extensions as counterexamples to Blackadar's conjecture. Adam S\o rensen proved in \cite{apws:ccb} that there are also non-split counterexamples to Blackadar's question. 

\subsection{Corners}

Blackadar proved in \cite{bb:stc} that any unital and full corner of a semiprojective $C^*$-algebra must again be semiprojective.
As our final order of business, we disprove by example (essentially, the same as the one given in the previous section) that Blackadar's conjecture that passage of semiprojectivity to corners remains automatic without the unitality condition. Intriguingly, we may instead prove that semiprojectivity passes from unital graph algebras to its unital corners, i.e., that Blackadar's result -- in this particular case -- remains true without the fullness condition. This will allow a new characterization of semiprojectivity in this case.

\begin{example}\label{bbwasagainnotentirelyright}
We saw in Example \ref{recexxiii}
that an ideal $I\gidealof C^*(\myii)$ was given by the graph
\[\raisebox{-2cm}[0.3cm]{}
\xymatrix@C-=1mm{
&&&\bullet&&&&&&&\\
&&&\bullet\ar@(lu,ld)[]\ar@(l,d)[]\ar[u]_(0.2){w_0}_(0.8){v_0}\ar[ul]\ar[ull]\ar[ulll]^(0.6){\dots}&&&\circ\ar[lll]\ar[u]^(0.2){w_1}\ar[ur]\ar[urr]\ar[urr]\ar[urrr]_(0.6){\dots}&&&
}
\]
and appealing to Theorem \ref{cornerisgraph} with $V=\{w_0,w_1,v_1\}\cup s(\myii^1w_1)$ we get a full corner $p_VIp_V$ which is the graph algebra given by
\[
\raisebox{-2cm}[0.3cm]{}
\xymatrix@C-=1mm{
{\bullet}&&&&{}&&&\\
\bullet\ar@(lu,ld)[]\ar@(l,d)[]\ar[u]_(0.2){w_0}_(0.9){v_0}&&&&\circ\ar[llll]\ar[u]^(0.2){w_1}\ar[ur]\ar[urr]\ar[urrr]_(0.6){\dots}
}
\]
seen also in Example \ref{recexxiii} to define an ideal of $C^*(\myv)$. By our observations in Example \ref{recexxiv}, $p_VIp_V$ fails to be semiprojective whilst $I$ is semiprojective. 
\end{example}

\begin{proposition}\label{cornissp}
When $E^0$ is finite and $C^*(E,\mys)$ is semiprojective, then so is $pC^*(E,\mys)p$ for any  $p\in C^*(E,\mys)$.
\end{proposition}
\begin{proof}
Fix a model $(V,\XV)$ with $[\pVX]\leq [p]\leq n[\pVX]$. We may assume without loss of generality that the model satisfies the conditions of Proposition \ref{cornerisgraph} because of Proposition \ref{passto}, and since by \cite{bb:stc} semiprojectivity of $p C^*(E,\mys)p$ will follow form semiprojectivity of $\pVX C^*(E,\mys)\pVX$ we see that we only need to check that $C^*(F,\mys\cap F)$ is semiprojective, where $F$ is the subgraph given by restricting to 
a vertex set $F^0$ of the form $H\sqcup B$ where $H$ is hereditary in $E$, and $B$ is a set of sinks in $F$ which are not mutually connected.

By Theorem \ref{Thm:finite} we know that for any $v\in E^0$, ${\Omega_v}$ does not have (FQ) in $E^0$ relative to $\mys\cap F$. 
For any $v\in F^0$, we need to check that $\{w\in F^0\mid $ does not have (FQ) in $F$ relative to $\mys$. Note first that $\widetilde{\Omega_v}=\{w\in F^0\mid |wF^1v|=\infty\}$ is always contained in $H$, since no edges end in $B$. And since $H$ is hereditary in $E^0$, we have that $\widetilde{\Omega_v}=\Omega_v$, and that the set does not have (FQ) even in $F$. The desired conclusion follows by Theorem \ref{Thm:finite} again.
\end{proof}

We then add

\begin{theorem}\label{Thm:finiteII}
The conditions (i)--(v) of Theorem \ref{Thm:finite} are equivalent to
\begin{enumerate}[(i)]\addtocounter{enumi}{5}
\item Any corner $p C^*(E,\mys)p$ with $p\in C^*(E,\mys)$ is semiprojective
\item For any ideal $I\gidealof C^*(E,\mys)$, any corner $p(C^*(E,\mys)/I)p$ with $p\in C^*(E,\mys)/I$ is semiprojective
\item No pair of ideals $I\gidealof J\gidealof C^*(E,\mys)$ has $J/I$ Morita equivalent to $\K^\sim$ or $(C(\T)\otimes \bK)^\sim$
\end{enumerate}
\end{theorem}
\begin{proof}
We get equivalence of (i), (vi) and (vii) by Propositions \ref{qissp} and \ref{cornissp}. Supposing that  (viii) fails, take a full projection $p$ in $J/I$. Then $p(J/I)p$ is Morita equivalent to a non-semiprojective $C^*$-algebra by assumption, and hence not itself semiprojective, proving that (vii) fails. 

We prove that (viii)$\Longrightarrow$ (iv) by contraposition. When $\Omega_v$ has (FQ) we define an ideal $I$ as in the proof of Proposition \ref{nec} using an admissible pair $(E\backslash  F,E^0\reg)$ where $F=s(xE^\dagger)$ is given by  a vertex $x$ in $E^0_1\cup (E^0_0\setminus \mys)$. Inside $C^*(E,\mys)/I\simeq C^*(F)$ we consider the model projection $q=p_{V,X}$ with
$V=\{v\}\cup xF^*\Omega _v$ and
\[
X_v=\{e\in vF^1\mid |s(e)F^1v|<\infty\}
\]
and see by Propositions \ref{passto} and \ref{cornerisgraph} that $qC^*(F)q$ is on the desired form, as illustrated in Figure \ref{likethis}. 
\end{proof}

We were lead to realize the veracity of (viii), and consequently of (vi) and (vii) by extensive experimentation as explained in Remark \ref{alltwobytwo}, cf.\ \cite{serj:ciugc}.

\begin{corollary}\label{oneideal}
The following  are equivalent for a unital graph algebra $C^*(E)$
with precisely one ideal: 
\begin{enumerate}[(i)]
\item $C^*(E)$ is semiprojective
\item $C^*(E)$ is not AF
\end{enumerate}
\end{corollary}

\begin{figure}
\begin{center}
\begin{tabular}{ccc}
\makebox[3.2cm][t]{\includegraphics[width=3cm]{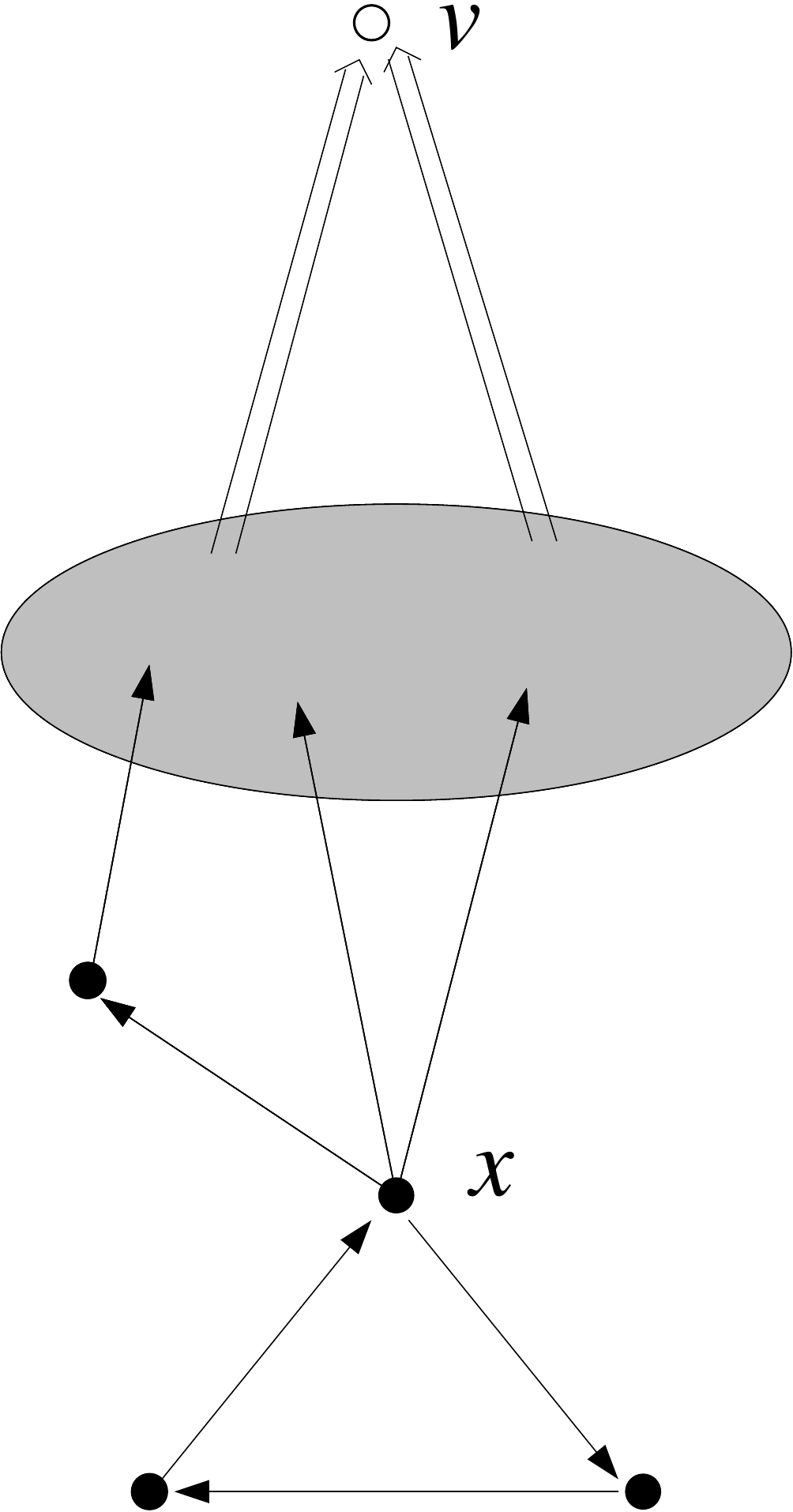}\rule{0mm}{3cm}}&\includegraphics[width=3cm]{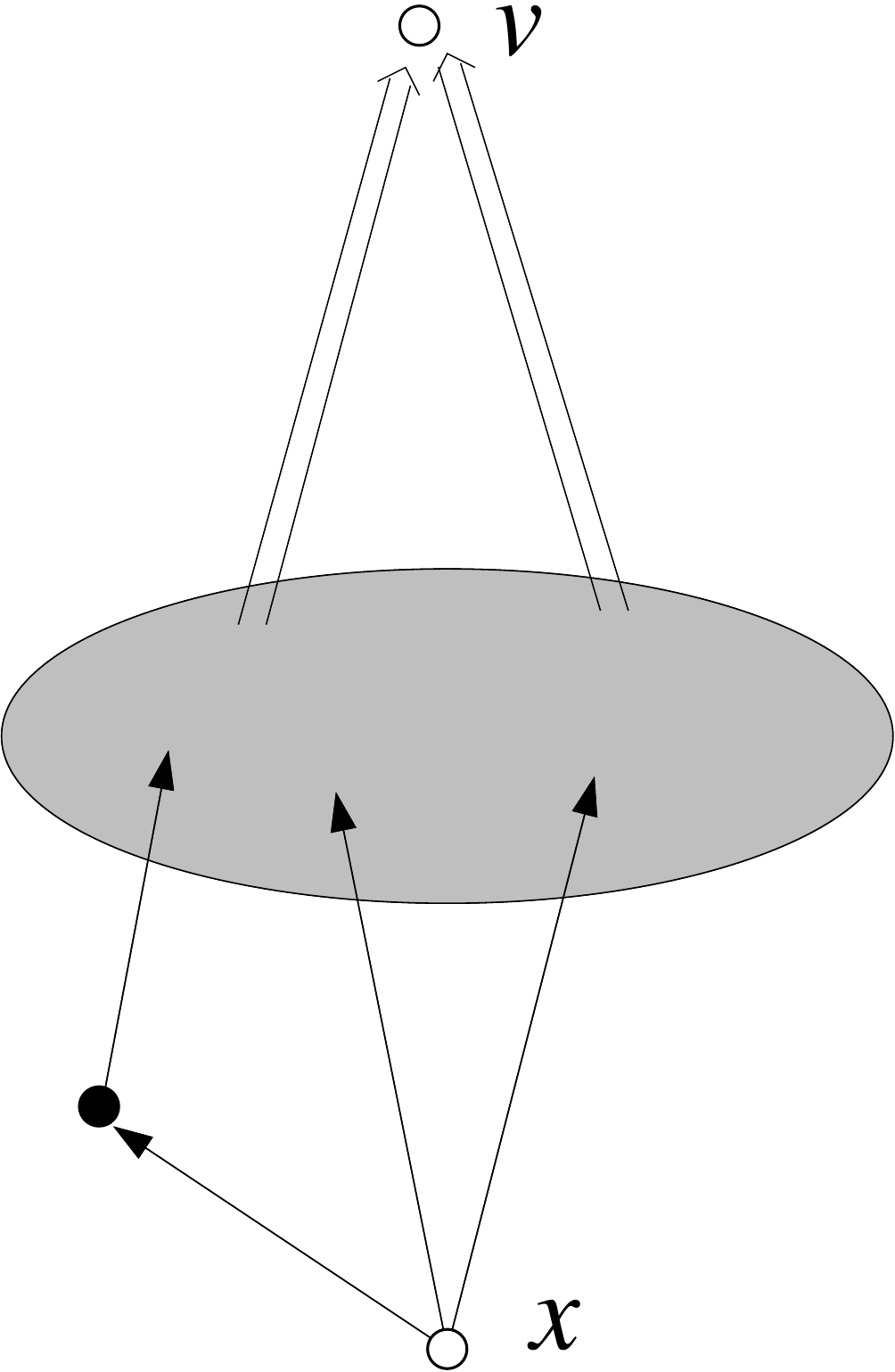}\\
\end{tabular}
\end{center}
\caption{Configurations in the non-semiprojective case}\label{likethis}
\end{figure}

As seen in \cite{setkermt:iagc}, any unital graph algebra with a single non-trivial ideal which is AF is isomorphic to a graph algebra given by 
\[
\xymatrix{
\circ\ar[r]&\bullet\ar[r]&\cdots\ar[r]&\bullet\ar@{=>}[r]&\circ\ar[r]&\bullet\ar[r]&\cdots\ar[r]&\bullet}
\]
with $m\geq 0$ edges to the left of the infinite emitter and $n\geq 0$ edges to the right of the infinite reciever.
There are many  graph $C^*$-algebras with precisely two non-trivial ideals which are neither AF nor semiprojective.  One example is given by the graph given by $\left[\begin{smallmatrix}2& 0\\ \infty& 0\end{smallmatrix}\right]$ from Example \ref{alltwobytwo}.

\newcommand{\etalchar}[1]{$^{#1}$}
\providecommand{\bysame}{\leavevmode\hbox to3em{\hrulefill}\thinspace}
\providecommand{\MR}{\relax\ifhmode\unskip\space\fi MR }
\providecommand{\MRhref}[2]{%
  \href{http://www.ams.org/mathscinet-getitem?mr=#1}{#2}
}
\providecommand{\href}[2]{#2}

\end{document}